\documentclass[10pt,a4paper]{article}

\usepackage{amsmath,amssymb,amsthm,mathrsfs}
\usepackage{stmaryrd}  
\usepackage{graphicx}
\usepackage{xcolor}
\usepackage{soul}
\usepackage{cancel}
\usepackage{changes}
\newcommand{\stkout}[1]{\ifmmode\text{\sout{\ensuremath{#1}}}\else\sout{#1}\fi}
\setdeletedmarkup{\stkout{#1}}
\usepackage{stmaryrd}

\usepackage{fullpage}
\usepackage{todonotes}
\usepackage{hyperref}
\definechangesauthor[name={Xiaonyu}, color=blue]{XN}

\newlength{\defbaselineskip}
\newcommand{\setlinespacing}[1]%
           {\setlength{\baselineskip}{#1 \defbaselineskip}}

\theoremstyle{plain}
\newtheorem{theorem}{Theorem}[section]
\newtheorem{lemma}[theorem]{Lemma}
\newtheorem{proposition}[theorem]{Proposition}
\newtheorem{corollary}[theorem]{Corollary}

\theoremstyle{definition}
\newtheorem{definition}[theorem]{Definition}

\newtheorem{ass}[theorem]{Assumption}

\theoremstyle{remark}
\newtheorem{remark}[theorem]{Remark}

\numberwithin{equation}{section}

\DeclareMathOperator*{\esssup}{ess\,sup}

 \allowdisplaybreaks
\begin{document}
\title{Stochastic Control Problems with Infinite Horizon and Regime Switching Arising in Optimal Liquidation with Semimartingale Strategies
}

\author{Xinman Cheng\footnote{Department of Applied Mathematics, The Hong Kong Polytechnic University, Kowloon,
		Hong Kong; email:xinman.cheng@connect.polyu.hk} \quad Guanxing Fu\footnote{Department of Applied Mathematics and Research Centre for Quantitative Finance, The Hong Kong Polytechnic University, Kowloon, Hong Kong; email: guanxing.fu@polyu.edu.hk. The author acknowledges financial support from the Hong Kong Research Grants Council (GRF Grant No. 15218825) and The Hong Kong Polytechnic University (Grant No. P0045668).}  \quad  Xiaonyu Xia\footnote{College of Mathematics and Physics, Wenzhou University, Wenzhou 325035, China; email: xiaonyu.xia@wzu.edu.cn. The author acknowledges financial support from the Natural Science Foundation of Zhejiang Povince (Grant No. LMS25A010012), the National Natural Science Foundation of China (Grants No. 12101465 and 12271391), and The Hong Kong Polytechnic University (Grant No. P0045668) during her visit.}
}

\maketitle

\begin{abstract}
 We study an optimal control problem on infinite time horizon with semimartingale strategies, random coefficients and regime switching. The value function and the optimal strategy can be characterized in terms of three systems of backward stochastic differential equations (BSDEs) with infinite horizon. One of them is a system of linear BSDEs with unbounded coefficients and infinite horizon, which seems to be new in literature. We establish the existence of the solutions to these BSDEs by BMO analysis and comparison theorem for multi-dimensional BSDEs. Next, we establish that the optimal control problem is well posed, in the sense that the value function is finite and the optimal strategy—when it exists—is unique. This is achieved by reformulating the cost functional as the sum of a quadratic functional and the candidate value function. The reformulation crucially relies on the well‑established well‑posedness results for systems of BSDEs. Finally, under additional assumptions, we obtain the unique optimal strategy.
\end{abstract}

{\bf AMS Subject Classification:} 93E20, 91B70, 60H30

{\bf Keywords:}{ infinite-horizon stochastic control, system of BSDEs with stochastic Lipschitz driver, regime switching, semimartingale strategy     }

\section{Introduction}

 Let $(\Omega, \mathcal F, \mathbb{P})$ be a fixed complete probability space on which we define a Brownian motion $W$ and a continuous-time stationary Markov chain $\alpha$ valued in a finite state space $\mathcal M = \left\{1, 2,\dots, \ell\right\}$ with $\ell \geq 1$. The Markov chain has a generator matrix $(q^{ij})_{\ell \times \ell}$ with $q^{ij} \ge 0$ for $i \neq j$ and $\sum^\ell_{j=1} q^{ij} = 0$ for each $i \in\cal M$. 
 We assume $W$ and $\alpha$ are independent of each other. We denote by $\mathbb F$ and $\mathbb G$ the augmented filtration generated by $(W,\alpha)$ and $W$, respectively.

In this paper, we study the following infinite-horizon stochastic control problem:
\begin{equation}\label{cost-inf-continuous}
	J(X)=\mathbb{E}\left[\int_{0}^\infty e^{-\phi s}\left(-Y_{s-}\,dX_s + \frac{\gamma_s}{2}\,d[X]_s - \sigma_s^{\alpha_s} \,d[X,W]_s+\lambda_s^{\alpha_s} X_s^2\,ds \right) \right]\rightarrow \min \quad \textrm{over }X,
\end{equation}
subject to the state dynamics
\begin{equation}\label{state-inf-continuous}
	\left\{ \begin{split}
		dY_s&=-\rho_s^{\alpha_s} Y_s\,ds-\gamma_s\,dX_s+\sigma_s^{\alpha_s}\,dW_s-d[\gamma,X]_s,\\
		X_{0-}&=x_0,\quad \,Y_{0-}=y_0,\quad \alpha_{0-}= i_0,
	\end{split} \right.
\end{equation}	
where $\gamma$ follows the dynamics
\begin{equation}\label{eq:gamma}
	d\gamma_s=\gamma_s\left(\mu_s  \,ds+\sigma_{1,s} \,dW_s\right),\quad \gamma_0> 0.\footnote{ Here we assume that the coefficients $\mu$ and $\sigma_1$ in this SDE \eqref{eq:gamma} are independent of the Markov chain $\alpha$ purely for simplicity of presentation. Since relaxing this assumption would introduce several extra terms into the candidate BSDE systems. While these additional terms would serve only to impose more restrictive conditions on the coefficients, rather than to change the underlying mathematical methodology.}
\end{equation}
The minimization in \eqref{cost-inf-continuous} is taken over all c\`adl\`ag semimartingales subject to some integrability conditions that will be specified later.

The control problem \eqref{cost-inf-continuous}-\eqref{state-inf-continuous} without regime switching and with constant $\gamma$ has been studied in \cite{CFX-2024}, motivated by an optimal liquidation problem with semimartingale strategies and infinite horizon, where $X$ is the position of the player and $Y$ is the price deviation process. 
Each term in the cost functional \eqref{cost-inf-continuous} is well interpreted by a micro-foundation; refer to \cite[Section 2.1]{CFX-2024}  and \cite[Section 4.1]{MuhleKarbe} for details. While optimal trade execution typically focuses on short horizons, Almgren \cite{almgren2012optimal} and Sch\"oneborn \cite{Schoneborn15} showed that removing temporal constraints fundamentally alters the optimal strategy, with certain model parameters even precluding full liquidation in finite time. Prompted by this funding, our recent work \cite{CFX-2024} investigates how strategies transform over extended horizons by analyzing the long-term behavior of a stochastic control problem with semimartingale strategies and external flows. 
 Similar models as \eqref{cost-inf-continuous}-\eqref{state-inf-continuous} without regime switching and with finite horizon are studied in Ackermann et al. \cite{AKU-2021, AKU-2022,AKU-2025} in different settings, with and without risk aversion, single asset or multiple asset.

The candidate value function and the optimal strategy can be characterized by three systems of infinite‑horizon BSDEs, one of which is a linear system with stochastic Lipschitz drivers that appears to be new to the literature. Our first contribution is to establish an existence result for a general class of infinite‑horizon BSDE systems with stochastic Lipschitz drivers, which is of independent interest.
To this end, we first truncate both the unbounded coefficients and the infinite time horizon. The resulting BSDE system has Lipschitz drivers on a finite horizon, and its well‑posedness follows from standard results in the literature (see, e.g., \cite[Theorem 4.1]{Briand03}). However, the corresponding solutions lie only in standard solution spaces, and the associated a priori estimates depend explicitly on the time horizon. This dependence prevents a direct passage to the infinite‑horizon limit. To overcome this difficulty, we construct a time‑weighted solution space by carefully choosing a discount factor so that the resulting a priori estimates are uniform with respect to the truncation index. To verify that solutions of the truncated system belong to this weighted space, we compare the system with a one‑dimensional BSDE, rather than an ODE as is commonly done in the regime‑switching literature (see, e.g., \cite{fu2025system,Hu-Liang-Tang}). This comparison requires a delicate BMO analysis, including the use of energy inequalities and reverse Hölder inequalities. We then let the truncation index tend to infinity and verify that the limit solves the original infinite‑horizon system.

Our second contribution is the resolution of the stochastic control problem \eqref{cost-inf-continuous}–\eqref{state-inf-continuous}. We first prove that the value function is finite, a fact that is not immediate from the original formulation of the cost functional \eqref{cost-inf-continuous}. Our approach begins with the finite‑horizon counterpart, where the cost can be rewritten as the sum of a nonnegative quadratic term and the candidate value function. The infinite‑horizon cost is then shown to arise as the limit of this decomposition, with convergence relying crucially on the well‑posedness results established in the first contribution. During the rewriting, to apply It\^o's formula for processes with regime switching (see \cite[Lemma 4.3]{Hu-Liang-Tang}), we must ensure that the relevant processes only jump as a result of regime‑switching. Motivated by \cite{AKU-2021,AKU-2022}, we reexpress the cost functional in terms of the process $P:=X+\frac{1}{\gamma}Y$, referred to as the scaled hidden deviation process. In our setting, all jumps of $P$ originate from the regime process $\alpha$. The existence of an optimal strategy is obtained under additional assumptions. The main challenge lies in proving admissibility of the candidate strategy. Again inspired by \cite{AKU-2021,AKU-2022}, we estimate a process related to 
$P$ rather than $X$ directly. Unlike \cite{AKU-2021}, the corresponding SDE involves unbounded linear coefficients. We eliminate this difficulty via a suitable linear transformation, yielding an SDE with coefficients expressed through a stochastic exponential, for which effective estimates can be derived.

Our work can be compared with two strands of the literature.
First, \cite{hu2024non} studies a linear–quadratic control problem with regime switching on a finite horizon. Their approach is based on a contraction mapping argument, which can be extended to the infinite‑horizon setting, albeit under the additional assumption that a certain discount factor must dominate a quantity related to the transition rates $q^{ij}$. In contrast, our approach to the analysis of general systems of linear BSDEs on an infinite horizon does not rely on such assumptions. Further discussion and technical details can be found in the first author’s PhD thesis\cite{xinmanphd}. Second, our work is related to the literature on multidimensional BSDEs with infinite horizons and unbounded coefficients; see, for example, \cite{li2019general,li2023bsdes,li2024weighted,lishun2020p}. In particular, \cite{li2019general} establishes well‑posedness for multidimensional BSDEs under a weak stochastic monotonicity condition, a general growth condition in $y$, and a stochastic Lipschitz condition in 
$z$. The results in \cite{li2023bsdes} refine this framework by integrating continuous dependence and comparison theorems within a unified proof. More recently, \cite{li2024weighted} develops existence results in weighted 
$L^2$-spaces with random terminal times under similar stochastic monotonicity and Lipschitz conditions. A common feature of these works is the crucial assumption that
$\int_0^\infty (|u_t| + v_t^2) \,dt < \infty$ a.s.,
where $u$ and 
$v$ denote the (possibly stochastic) coefficients of $y$ and 
$z$, respectively; see \cite[Section 2]{li2019general}, \cite[Assumption (I3)]{li2023bsdes}, \cite[Inequality (1.10)]{li2024weighted}, and \cite[Assumptions (A4)–(A5)]{lishun2020p}. In our setting, however, a key observation is that the coefficient of $y$ in the infinite‑horizon BSDE depends explicitly on the transition rates $q^{ij}$, causing this integrability condition to fail. In addition, an earlier work by Papapantoleon et al.~\cite{papapantoleon18} studies a broader class of multidimensional BSDEs with unbounded random time horizons driven by general martingales, allowing for stochastic discontinuities in the filtration. Their model is formulated within a sophisticated and highly general framework, incorporating jumps and a general filtration. Nevertheless, when one abstracts from these structural features to focus on the core mechanism, their approach encounters a similar coefficient integrability issue to that identified in \cite{li2019general,li2023bsdes,li2024weighted,lishun2020p} when applied to our setting. As a result, these existing frameworks cannot be directly applied without imposing additional restrictive assumptions. Filling this gap requires the development of new methodologies tailored to our systems of infinite‑horizon BSDEs with unbounded coefficients. We also note that techniques developed for multidimensional BSDEs with unbounded coefficients on finite horizons generally do not extend to our infinite‑horizon framework. For example, \cite{delbaen2010harmonic} studies multidimensional SDEs and BSDEs on finite horizons, driven by semimartingales and continuous local martingales, respectively, under suitable sliceability conditions in the BMO space for unbounded coefficients. However, these conditions are insufficient for our infinite‑horizon analysis. The underlying difficulty again stems from the constant transition rates 
$q^{ij}$, which render the coefficient of 
$y$ non‑sliceable in the BMO space.

Before presenting the section‑by‑section summary, we briefly review some existing results on optimal liquidation under regime switching; see \cite{bian2016optimal,fu2025system,pemy2006optimal,pemy2008liquidation,siu2019optimal}. Among these works, \cite{pemy2006optimal} employs regime switching to model stock price dynamics and studies an optimal liquidation problem formulated as an optimal stopping problem. In \cite{pemy2008liquidation}, the authors further allow the regime‑switching intensities to depend on the trading intensity and analyze an infinite‑horizon liquidation problem using viscosity solution techniques. The work \cite{bian2016optimal} considers a finite‑horizon liquidation model incorporating both permanent and temporary price impact, again studied via a viscosity solution approach. In a discrete‑time setting, \cite{siu2019optimal} investigates an optimal liquidation problem in which regime switching is used to model stochastic order book resilience. More recently, \cite{fu2025system} analyzes a system of singular BSDEs arising from a liquidation problem, but restricts attention to the finite‑horizon case.
All of these studies differ from the focus of the present work, which centers on infinite‑horizon liquidation problems and their characterization through systems of BSDEs with unbounded coefficients.

The remainder of the paper is organized as follows. We conclude the introduction by introducing notation and stating the standing assumptions. Section \ref{sec:BSDEs} formulates the candidate BSDE systems associated with the control problem \eqref{cost-inf-continuous}–\eqref{state-inf-continuous}. Motivated by these systems, Section \ref{sec:general-BSDE} is devoted to the analysis of a general class of infinite‑horizon BSDE systems with unbounded coefficients and establishes the corresponding well‑posedness results. Finally, Section \ref{sec:control} proves the solvability of the original control problem.

\subsection{Notation, convention, assumption and admissible space}
\textbf{Notation.}
Denote the Euclidean norm by $|\cdot|$. Let $\mathcal S$ be an Euclidean space and $\mathcal P_{\mathbb H}(\mathcal S)$ be the space of all $\mathbb H\times[0,\infty)$-progressively measurable $\mathcal S$-valued stochastic processes, with $\mathbb H=\mathbb F$ and $\mathbb G$, respectively. For each $p>1$ and $K\in\mathbb R$, we define the following spaces:
\begin{equation*}
	\begin{split}
 	  M^{p,K}_{\mathbb H}(0,\infty;\mathcal S) =&~ \left\{   v\in\mathcal P_{\mathbb H}(\mathcal S):     \|v\|_{M^{p,K}}	:=  \mathbb E\left[\left(\int_0^\infty e^{-2Kt} |v_t|^2\, dt\right)^{p/2}\right]^{1/p}<\infty 				\right\},\\
	S^{p,K}_{\mathbb H}(0,\infty;\mathcal S)=&~\left\{    v\in\mathcal P_{\mathbb H}(\mathcal S):  	\| v \|_{p,K}	:=	\mathbb E\left[\sup_{t\geq 0}e^{-pKt}|v_t|^p\right]^{1/p}<\infty	   \right\},\\
	L^{\infty,K}_{\mathbb H}(0,\infty;\mathcal S)=&~\left\{  v\in\mathcal P_{\mathbb H}(\mathcal S):   
\|v\|_{\infty,K}:= \mathop{\text{ess sup}}\limits_{(t,\omega)\in [0,\infty)\times\Omega} e^{-Kt}|v_t| <\infty \right\}.
	\end{split}
\end{equation*}
Let $\mathcal T$ be the space of all $[0,\infty)$-valued stopping times. We say the process $\left(\int_0^t e^{-Ks} v_s \,dW_s\right)_{t\geq 0}$ is a BMO martingale, if it satisfies 
\[
	\|v\|_{\text{BMO},K}:=\sup_{\tau\in\mathcal T}\left(\mathop{\text{ess sup}}\limits_{\omega\in \Omega}\mathbb E_\tau\left[\int^\infty_\tau e^{-2Ks} |v_s|^2\,ds\right]\right)^\frac{1}{2}<\infty,
\]
where for notational convenience, we use $\mathbb E_t[\cdots]$ to denote conditional expectation with respect to any filtration; the intended filtration will be clear from context. The space of all such processes $v\in \mathcal P_{\mathbb H}(\mathcal S)$ is denoted by $H^{2,K}_{\text{BMO},\mathbb H}(0,\infty;\mathcal S)$.

 Note that the constant $K$ in the above definitions of spaces can be either positive or negative.

The Dol\'eans-Dade stochastic exponential of $\int_0^\cdot v_s\,dW_s$ is defined as
\begin{equation*}
	\mathcal{E}\left(\int_0^\cdot v_s\,dW_s \right)_{t} = \exp\left\{\int_{0}^{t} v_s dW_s - \frac{1}{2} \int_{0}^{t} |v_s|^2 ds\right\},
\end{equation*}
which is denoted by $\mathcal E(v)$ for notational convenience.

\textbf{Convention.}
\begin{enumerate}
	\item When the filtration, time horizon, and base space are clear from context, or when we do not want to emphasize the base space, we omit them from the space notation. For instance, we write $L^{\infty,K}$ and $H^{2,K}_{\mathrm{BMO}}$ as shorthand for $L^{\infty,K}_{\mathbb{H}}(0, \infty; \mathcal{S})$ and $H^{2,K}_{\mathrm{BMO}, \mathbb{H}}(0, \infty; \mathcal{S})$, respectively, when $\mathbb{H}$, $[0, \infty)$, and $\mathcal{S}$ are understood from context.
	
	\item The symbol $c$ denotes a positive constant that may vary from line to line, but is always independent of any sequence index tending to infinity.
	
	\item In subsequent estimates, we often omit vector indices by appealing to the Euclidean norm. For example, for a vector \( v := (v^1, \dots, v^\ell) \), we frequently use the inequality
	\[
	|v^i| \leq |v| = \left( \sum_{j=1}^\ell (v^j)^2 \right)^{1/2}
	\]
	without explicitly stating it.
\end{enumerate}
The following standing assumptions are assumed to hold throughout the paper.	
\begin{ass}\label{ass:standing} 
	The discount factor $\phi$ is a nonnegative constant. For each $i\in\cal M$, assume that $\sigma^i$, $\rho^i$, $\lambda^i$, $\mu$ and $\sigma_1$ belong to $\mathcal P_{\mathbb G}(\mathbb R)$. Moreover,
	there exist positive constants $\beta$, $\hat c$, $L$ and $\epsilon$ such that for any $t\in[0,\infty)$ and $ i\in\cal M$, it holds
\begin{enumerate}

	\item[(i)]	 $|\rho^i_t| \leq  \hat c$ and $0\leq \frac{\lambda^i_t}{\gamma_t}\leq \hat c$,

	\item[(ii)] $|\mu_{t}|\leq \hat c e^{-Lt}$ and $|\sigma_{1,t}|\leq \hat ce^{-Lt}$ and $|\widetilde\sigma^i_t|\leq \hat c e^{-\frac{\beta t}{2}}$ where $\widetilde\sigma^i_t:= e^{-\frac{\phi t}{2}}\sigma^i_t$,  

	\item[(iii)] $\frac{\phi}{2}+\rho^i_t+\frac{\mu_t}{2}-\frac{\sigma_{1,t}^2}{2}\geq \epsilon$, $\rho^i_t+\mu_t-\sigma_{1,t}^2\geq 0$ and  $\frac{\phi}{2} +  \frac{
		\frac{\lambda_t^{i}}{\gamma_t}(\rho_t^{i}+\mu_t) }{  \frac{\mu_t}{2} + \rho_t^{i} +
		\frac{\lambda_t^{i}}{\gamma_t} + \frac{\phi}{2}} \geq \epsilon$.
\end{enumerate}
\end{ass}
{\bf The space of admissible strategies.}
The admissible space for \eqref{cost-inf-continuous}-\eqref{state-inf-continuous}, denoted by $\mathscr A$, is defined as all $X\in\mathcal P_{\mathbb F}(\mathbb R)$ satisfying 
\begin{enumerate}
	\item[(i)] $X$ is a c\`adl\`ag semimartingale;
	\item[(ii)] the following integrability is satisfied by $\widetilde X$
	\begin{equation}\label{integrability:X}
		\mathbb E\left[\sup_{t\geq 0} e^{2\varphi t}\gamma_t \widetilde{X}_t^2 \right]<\infty,
	\end{equation}
respectively,
\begin{equation}\label{integrability:P}
	\mathbb E\left[\sup_{t\ge 0} e^{4\varphi t} \gamma_t^2 \widetilde{P}_t ^4\right]<\infty,
\end{equation}
\end{enumerate}
where $0<\varphi<\epsilon\wedge \frac{\beta}{8}$, $\widetilde X_s:= e^{-\frac{\phi}{2}s}	X_s$, $\widetilde Y_s:= e^{-\frac{\phi}{2}s}Y_s$ and $\widetilde P_s:=\widetilde X_s+\frac{1}{\gamma_s}\widetilde Y_s$.

Due to the one-to-one correspondence between $X$ and $\widetilde X$, we may abuse the notation by writing $\widetilde X\in\mathcal A$ for simplicity, if $X$ satisfies (i) and $\widetilde X$ satisfies (ii).


\section{Candidate BSDE systems}\label{sec:BSDEs}
In this section, we formulate the system of backward stochastic differential equations (BSDEs) associated with the control problem under study. We begin by reformulating the optimal control problem \eqref{cost-inf-continuous}–\eqref{state-inf-continuous} in dynamic matrix form:
\begin{equation}\label{cost-t-infty}
	J^{}(t, \mathcal X_{t-},i ;\widetilde X)=\mathbb{E}_t\left[ \int_{t}^{\infty} \left(\mathcal X_{s}^\top\mathcal L d\widetilde{X}_s+ \mathcal R_s\,d[\widetilde{X}]_s +	\mathcal L^\top\mathcal{D}_{s}^{\alpha_s} \,d[\widetilde{X},W]_s+\mathcal X_{s}^\top \mathcal Q_s^{\alpha_s}\mathcal X_{s}\,ds \right)  \right],
\end{equation}
subject to the controlled state dynamics
\begin{equation}\label{state-matrix}
	d \mathcal X_{s} =\mathcal{H}_s^{\alpha_s}\mathcal X_s d s+\mathcal{D}_{s}^{\alpha_s} d W_{s}+\mathcal{I}_s d \widetilde X_{s}+\mathcal P_s d[\widetilde X, W]_s, \quad s \in [t, \infty),
\end{equation}
where $\mathcal X=(\widetilde X,\widetilde Y)^\top$. 
The coefficient matrices are defined as follows: 
\begin{equation}\label{eq:matrix-coeff}
	\begin{split}
		\mathcal L=&~ \begin{pmatrix}
			0 & -1 
		\end{pmatrix}^{\top},\qquad  \mathcal R_s=\frac{\gamma_s}{2},\qquad 	\mathcal{D}_{s}^{\alpha_s}= \begin{pmatrix}
			0 & \widetilde\sigma_s^{\alpha_s}  
		\end{pmatrix}^{\top},\qquad 
		\mathcal{I}_s    = \begin{pmatrix}
			1 & -\gamma_s 
		\end{pmatrix}^{\top} ,\\ 
		\mathcal Q_s^{\alpha_s}=&~ \begin{pmatrix}
			\lambda_s^{\alpha_s} & -\frac{\phi}{4}  \\
			-\frac{\phi}{4}  & 0
		\end{pmatrix},\qquad  
		\mathcal{H}_s^{\alpha_s} = \begin{pmatrix}
			0 & 0  \\
			-\frac{\phi}{2}\gamma_s & -\rho_s^{\alpha_s} -\frac{\phi}{2}
		\end{pmatrix}, \qquad 	\mathcal{P}_{s}= \begin{pmatrix}
			0 & -\gamma_s\sigma_{1,s}
		\end{pmatrix}^{\top}
		.
	\end{split}
\end{equation}
Using either a discrete-time approximation approach (see \cite{AKU-2021,FHX-2023}) or a transformation to a linear-quadratic control framework (see \cite{AKU-2022}), we now present three systems of infinite-horizon BSDEs that will play a fundamental role in establishing well-posedness and in characterizing both the value function and the optimal strategy.

$\bullet$  Let \( A := (A^i)_{i=1,\dots,\ell} \) be a vector of symmetric \( \mathbb{R}^{2 \times 2} \)-valued processes satisfying \( A^i_{11} = \gamma A^i_{21} \) and \( A^i_{12} = \gamma A^i_{22} + \frac{1}{2} \) for each \( i \in \mathcal{M} \). Define \( \bar{A}^i := A^i_{12} \), and suppose \( (\bar{A}^i)_{i \in \mathcal{M}} \) satisfies the following system of BSDEs on \( [0, \infty) \):
	\begin{equation}\label{BSDE:A-inf}
		\begin{aligned}
			-d\bar{A}^{i}_{s}=&~\left\{\sum^\ell_{j=1}q^{ij}\bar A_s^{j} + ( \mu_s -\phi) \bar A^{i}_{s} +\sigma_{1,s}Z^{\bar A,i}_s+\frac{\lambda_s^i}{\gamma_s}-\frac{\gamma_s}{a^i_s}\left((\mu_s + \rho_s^i)\bar A_s^{i}+\sigma_{1,s}Z^{\bar A,i}_s  +\frac{\lambda_s^i}{\gamma_s}\right)^2\right\}\,ds\\
			&~-Z^{\bar A,i}_{s}\,dW_s,
		\end{aligned} 
	\end{equation}
	with the transversality condition 
	\begin{equation}\label{trans:A}
		\lim\limits_{T\to\infty}   \mathbb E\left[ e^{-pK_{\bar A}T}|\bar A^i_T|^p\right]=0 \quad \text{ for any } K_{\bar A}>0,\, p\ge 1, \, i\in\cal M.
	\end{equation}
		 Here, the auxiliary coefficient $a$ that will be used frequently throughout the paper is defined as 
	\begin{equation}\label{def:a-i}
		a_t^{i}=\gamma_t\sigma_{1,t}^2\bar{A}_t^{i}+\frac{\mu_t\gamma_t}{2}-\frac{\gamma_t \sigma_{1,t}^2}{2}+\gamma_t\rho_t^i+\lambda_t^i+\frac{\phi\gamma_t}{2}.
	\end{equation}

$\bullet$ Let \( B := (B^i)_{i \in \mathcal{M}} = ((B^i_1, B^i_2)^\top)_{i \in \mathcal{M}} \) be a vector of \( \mathbb{R}^{2 \times 1} \)-valued processes satisfying \( B^i_1 = \gamma B^i_2 \). Define \( \bar{B}^i := B^i_1 \), and suppose \( (\bar{B}^i)_{i \in \mathcal{M}} \) satisfies the following system of BSDEs on \( [0, \infty) \):
	\begin{equation}\label{BSDE:B-inf}
		  \begin{aligned}
			-d\bar B^{i}_{s}=&~\left\{\sum^\ell_{j=1}q^{ij}\bar B_s^{j}-\frac{\phi}{2}\bar B^{i}_{s}+2\widetilde\sigma^i_s Z^{\bar A,i}_{s} -\frac{\gamma_s}{a^i_s} \left( (\mu_s+\rho_s^i)\bar A_s^{i}+\sigma_{1,s}Z^{\bar A,i}_s+\frac{\lambda_s^i}{\gamma_s}  \right) \right.\\
			&~\quad \times     \left(  (\rho_s^i +\mu_s-\sigma_{1,s}^2  ) \bar B^{i}_{s}+2  \sigma_{1,s}\widetilde\sigma_s^i\bar A_s^{i}- \sigma_{1,s}\widetilde\sigma_s^i +\sigma_{1,s}Z^{\bar B,i}_s \right)   \Bigg\}\,ds -Z^{\bar B,i}_{s}\,dW_s,
		\end{aligned} 
	\end{equation}
	with the transversality condition 
	\begin{equation}\label{trans:B}
		\lim\limits_{T\to\infty} \mathbb E\left[   e^{\frac{1}{4}\beta p T}|\bar B^i_T|^p \right]= 0 \quad 	\text{ for each }  p\geq 1 \text{ and }  i\in\mathcal  M.
	\end{equation}

$\bullet$ Let \( C := (C^i)_{i \in \mathcal{M}} \) be a vector of real-valued processes satisfying the following system of BSDEs:
	\begin{equation}\label{BSDE:C-inf}
		\begin{aligned}
			-d C^{i}_s=&~\left\{\sum^\ell_{j=1}q^{ij}C_s^{j} - \frac{1}{4a^i_s}\left(  (\rho_s^i  +\mu_s-\sigma_{1,s}^2   )\bar B^{i}_{s}+2 \sigma_{1,s}\widetilde\sigma_s^i\bar A_s^{i}- \sigma_{1,s}\widetilde\sigma_s^i + \sigma_{1,s}Z^{\bar B,i}_s \right)^2\right.\\
			&~\left.+(\widetilde\sigma_s^i)^2 \left(\frac{\bar A^{i}_{s}}{ \gamma_s}-\frac{1}{2\gamma_s}\right) -  \sigma_{1,s}\frac{\widetilde \sigma_s^i}{\gamma_s}\bar B_s^{i}+\frac{\widetilde\sigma^i_s}{\gamma_s}  Z^{\bar B,i}_{s}\right\}\,ds-Z^{C,i}_s\,dW_s,
		\end{aligned} 
	\end{equation}
 with the transversality condition 
 \begin{equation}\label{trans:C}
 		\lim\limits_{T\to\infty} \mathbb E\left[e^{\frac{\beta}{4}T}|C^i_T|^2\right]= 0\quad \text{for each }i\in\cal  M,
 \end{equation}
 where $\beta$ is the constant appearing in Assumption \ref{ass:standing}.

In Section \ref{sec:convergence-BC}, we will prove the systems of infinite horizon BSDEs \eqref{BSDE:A-inf}-\eqref{trans:C} admit one solution. 

\section{A general system of infinite horizon BSDEs with stochastic Lipschitz drivers}\label{sec:general-BSDE}

Motivated by the systems of BSDEs introduced in Section~\ref{sec:BSDEs}, particularly the infinite-horizon systems \eqref{BSDE:B-inf}-\eqref{trans:C}, we study a general class of infinite-horizon BSDEs with stochastic Lipschitz drivers. The systems \eqref{BSDE:B-inf}, \eqref{BSDE:C-inf} will appear as a special case of this general framework. The results presented in this section are of independent interest. Before stating the main theorem, we first state several preliminary lemmas.
\subsection{Preliminary lemmas} 
\begin{lemma}\label{lemma:zeta-BMO}
Assume $\zeta\in H^{2,K_{\zeta}}_{\text{BMO}}$, where $K_\zeta$ is a given constant.
For any $b_1\in(0,2)$, $b_2>0$ and $b_3>K_{\zeta}$, it holds that 
	\begin{equation}\label{inequ:zeta}
		\begin{split}
	\left\|  e^{-b_1 b_3\cdot} |\zeta_\cdot|^{b_1} \right\|_{b_2,exp} :=&~\esssup_{\omega\in\Omega,\tau\in\cal T}\mathbb E _\tau\left[\exp\left(b_2\int^\infty_\tau e^{-b_1b_3s}|\zeta_s|^{b_1}\,ds\right)\right] \\
	\leq&~ \frac{4}{3}\exp\left\{        \left(    \frac{1}{2b_1N^2}   \right)^{ -\frac{b_1q}{2} }    \frac{b_2^q}{   b_1q^2(b_3-K_\zeta)   }			\right\},
	\end{split}	
\end{equation}
	where  $q$ is the conjugate of $\frac{2}{b_1}$, and $N$ is any positive constant such that 
		\begin{equation}\label{inequ:zeta-BMO}
		\left\|\zeta\right\|_{\text{BMO},K_\zeta}\leq N.
	\end{equation}
\end{lemma}
\begin{proof}  
	Given $\zeta\in H^{2,K_{\zeta}}_{\text{BMO}}$ and \eqref{inequ:zeta-BMO},
by \cite[Theorem 2.2]{Kazamaki-1994} we have for each $\tau\in\mathcal T$
	\[
	\mathbb E_\tau\left[\exp\left(\frac{1}{4N^2}\int^\infty_\tau e^{-2K_{\zeta}s}|\zeta_s|^2\,ds\right)\right]\leq \frac{1}{1-\|\int^\cdot_0\frac{1}{2N}e^{-K_{\zeta}s}\zeta_s\,dW_s\|^2_{\text{BMO}}}\leq \frac{1}{1-\frac{1}{4}}=\frac{4}{3}.
	\]
	
	Thus, by Young's inequality we have 
	\[
	\begin{aligned}
		&~\sup_\tau\mathbb E_\tau\left[\exp\left\{ b_2 \int^\infty_\tau e^{-b_1b_3 s}|\zeta_s|^{b_1}\,ds\right\}\right]\\
		=&~\sup_\tau \mathbb E_\tau\left[\exp\left\{\int^\infty_\tau \left(\left(\frac{2}{b_1}\cdot \frac{1}{4N^2}\right)^{\frac{b_1}{2}}     e^{-b_1K_\zeta s} |\zeta_s|^{b_1}   \times     \left(\frac{2}{b_1} \cdot  \frac{1}{4N^2}    \right)^{-\frac{b_1}{2}}b_2 e^{-b_1(b_3-K_\zeta)s}\right)\,ds\right\}\right] \\
		\leq &~\sup_\tau \mathbb E_\tau\left[\exp\left\{\int^\infty_\tau \left(   \frac{1}{4N^2}e^{-2K_\zeta s}|\zeta_s|^2 +\frac{1}{q}\left(\frac{2}{b_1}\cdot \frac{1}{4N^2}\right)^{-\frac{b_1q}{2}} b_2^q e^{-b_1q(b_3-K_\zeta)s}\right)\,ds\right\}\right]\quad \left(\text{where }\frac{1}{q}+\frac{1}{2/b_1}=1\right)\\
		\leq &~
		\frac{4}{3}\exp\left\{        \left(    \frac{1}{2b_1N^2}   \right)^{ -\frac{b_1q}{2} }    \frac{b_2^q}{   b_1q^2(b_3-K_\zeta)   }			\right\}.
	\end{aligned}\]
\end{proof}

The following lemma can be found in \cite[Page 29]{Kazamaki-1994}.
\begin{lemma}[Energy inequality]\label{lemma:energy}
	Assume $Z\in H^{2,0}_{\text{BMO}}$. For any constant $n\in\mathbb N_+$, it holds that 
	\begin{equation}\label{eq:BMO}
		\esssup_{\omega,\tau}\mathbb E_\tau\left[\left(\int^\infty_\tau|Z_s|^2\,ds\right)^n\right]\leq n!\left\| Z \right\|^{2n}_{\text{BMO}}.
	\end{equation}
\end{lemma}
The following reverse H\"older inequality can be found in \cite[Theorem 3.1]{Kazamaki-1994}. 
\begin{lemma}[Reverse H\"older inequality]\label{lemma:reverse}
	Let $\Phi$ be the function on $(1, \infty)$ defined as 
	\begin{equation*}
		\Phi(x) = \left( 1 + \frac{1}{x^{2}} \log \frac{2x-1}{2(x-1)} \right)^{1/2} - 1,
	\end{equation*}
which is nonincreasing, $\lim_{x \searrow 1} \Phi(x) = \infty$ and $\lim_{x \nearrow \infty} \Phi(x) = 0$.
Let $Z\in H^{2,0}_{\text{BMO}}$ and $q_*=\Phi^{-1}(\|Z\|_{\text{BMO}})$. For any $1<q<q_*$, it holds that 
\[
		\esssup_{\tau\in\cal T}\mathbb E_\tau\left[ 	\left|\frac{\mathcal E(Z)_\infty}{\mathcal E(Z)_\tau}\right|^q		   \right]\leq \mathcal K(q,\| Z \|_{\text{BMO}}),
\]
where
	\begin{equation*}
	\mathcal K(q, O) = \frac{2}{1 - \frac{2(q-1)}{2q-1 } \exp\left\{ q^{2}(O^{2}+2O)\right\}  }.
\end{equation*}

\end{lemma}



\subsection{A general infinite-horizon BSDE system}\label{subsec:general-BSDE}

The goal of this section is to establish the existence of solutions to  the following infinite-horizon BSDE system:
	\begin{equation}\label{BSDE:Y-infinite}
		\left\{\begin{aligned}
			-dY^i_t=&\left[-\nu^i_t Y^i_t+\sum^\ell_{j=1}q^{ij}Y^j_t+\sum^\ell_{j=1}\kappa^{ij}_tY^j_t+(g^i_t)^\top Z^i_t+f^i_t\right]\,dt-(Z^i_t)^\top \,d\mathcal W_t,\\
			i\in&~ \cal M,
		\end{aligned}\right.
	\end{equation}
where $\mathcal W$ is a multidimensional Brownian motion.

The following assumptions on the coefficients of \eqref{BSDE:Y-infinite} are assumed to hold throughout Section \ref{subsec:general-BSDE}.
 \begin{ass}\label{ass:general-BSDE}
\begin{enumerate}
	\item[(i)]  Let $r\in(0,2)$, $\bar K\leq \bar c$ be fixed constants throughout this section\footnote{Notice that $\bar K$ and $\bar c$ can be either positive or negative.}. All coefficients are assumed to be $\mathscr F\times[0,\infty)$ progressively measurable, where $\mathscr F$ denotes the augmented natural filtration of the Brownian motion $\mathcal W$. 
	\item[(ii)] $g\in H^{2,0}_{\text{BMO}}$.
	
	\item[(iii)] Let $\kappa:=(\kappa^{ij})_{i,j\in\cal M}\in\mathcal P_{\mathscr F}(\mathbb R^{\ell\times\ell})$ satisfy $\kappa^{ij} \ge 0\text{ for }  i \neq j \footnote{ This condition is required for the monotonicity in comparison principle for multidimensional BSDEs. } \text{ and }|\kappa| \leq |\phi|^r$ where $\phi\in H^{2,K_\phi}_{\text{BMO}}$ for $K_\phi< 0$.
	\item[(iv)] Let  $\nu:=(\nu^i)_{i\in\cal M}\in\mathcal P_{\mathscr F}(\mathbb R^\ell)$ satisfy   $\bar K\leq\nu^{i}\leq \bar c  \text{ for each }i\in\cal M.$
	\item[(v)] 	Given a constant  $K_f<\bar K$,  $|f^{}|^{1/2}\in \bigcap_{p\geq 1}M^{p, K_f/2}$.

\end{enumerate}

\end{ass}

We will first establish the existence of solutions to a truncated version of \eqref{BSDE:Y-infinite}, and second, show that this solution converges to one solution to \eqref{BSDE:Y-infinite} with some transversality condition.

For each given constant $m>0$, define\footnote{We abuse the notation by letting $h_m$ denote functions on both $\mathbb R$ and $\mathbb R^\ell$, without confusion.} 
\[
	h_m (x) := \frac{\min \{| x |, m \}}{| x|} \cdot x \cdot \mathbf{1}_{\{ | x | \neq 0 \}}
\]
and consider the truncated version of \eqref{BSDE:Y-infinite}
	\begin{equation}\label{BSDE:general-k-P}
		\left\{\begin{aligned}
			-d Y^{m,i}_t=&~\left[-\nu^{i}_tY^{m,i}_t+\sum^\ell_{j=1}q^{ij} Y^{m,j}_t+\sum^\ell_{j=1}h_m(\kappa^{ij}_t) Y^{m,j}_t+h_m(g^{i}_t)^\top Z^{m,i}_t+ f^{i}_t\right]\,dt\\
			&~-(Z^{m,i}_t)^\top\,d\mathcal W_t,  \quad t\in[0,m),\\
			Y^{m,i}_t=&~0,\quad t\in[m,\infty),\\
			 i\in&~ \cal M.
		\end{aligned}\right.
	\end{equation}

The following lemma establishes both the existence of a solution to \eqref{BSDE:general-k-P} and an a priori estimate for that solution, uniformly in $m$.
\begin{lemma}\label{lemma:general-P}
	For each $m$, each $K\in (K_f,\bar K) $ and each $K'>K$, the system of infinite-horizon BSDEs \eqref{BSDE:general-k-P} admits a solution 
	\[
		(  Y^{m},  Z^{m}):=(  Y^{m,i},  Z^{m,i}  )_{i\in\mathcal M}   \in  \bigcap_{p\geq 1}S^{p, K}\times M^{2, K'}.
	\]  
	Moreover, for $Y^{m}$ we have the following estimate for any $p>2a$
	\begin{equation}\label{estimate:Y-Pik-2}
		\begin{split}
			\mathbb E\left[  \sup_{t\geq 0}\left| e^{-K t} Y^{m}_t\right|^p  \right] 
			\leq  c\mathcal K(q,\bar N)^\frac{p}{q} \left\| |\phi|^r\right\|^\frac{p}{2a}_{2a\ell,exp} \mathbb E\left[ \left( 	\int_0^\infty  e^{-K_fs} | f_s |\,ds	  \right)^p \right],
		\end{split}
	\end{equation}
	where $\bar N=\|   g \|_{\text{BMO}}$, $q>1$ and $a>0$ are two constants determined in the proof, and the function $\mathcal K$ is given in Lemma \ref{lemma:reverse}.
	For $Z^{m}$ we have the following estimate:
		\begin{equation}\label{estimate:Z-Pik-2}
			\begin{split}
					\|	  Z^{m}		\|^2_{M^{2,K'}}				
				\leq~c\| Y^{m} \|^2_{2,K}+ c\|Y^{m}\|^{2}_{4,K}\left(\|g\|^2_{\text{BMO}}+\left\|  \phi \right\|^{r}_{\text{BMO},K_\phi}\right)+c\| Y^{m} \|_{2,K}\left\|   |  f  |^{\frac{1}{2}} \right\|^2_{M^{4, K_f/2}}.
			\end{split}
	\end{equation}
If we further assume that $|f^{}|^{1/2}\in H^{2, K_f/2}_{\text{BMO}}$, then $(  Y^{m},  Z^{m})\in L^{\infty,K} \times H^{2,K'}_{\text{BMO}}$. We also have the following estimate
\begin{equation}\label{estimate:Y-Pik}
	\begin{split}
		\|Y^{m}\|_{\infty,K}\leq c \mathcal K(q,\bar N)^\frac{1}{q}\left\| |\phi|^r\right\|^\frac{1}{2a}_{2a\ell,exp}  \left\|	  |f|^{\frac{1}{2}}	\right\|^{2}_{\text{BMO},\frac{1}{2}K_f},
	\end{split}
\end{equation}
and
		\begin{equation}\label{estimate:Z-Pik}
	\begin{split}
			\|	  Z^{m}		\|^2_{\text{BMO},K'}				
		\leq&~c\| Y^{m} \|^2_{\infty,K} +  c\|Y^{m}\|^2_{\infty,K}\left(\|g\|^2_{\text{BMO}}+\left\|  \phi \right\|^{r}_{\text{BMO},K_\phi}\right)+ c \left\|   |  f  |^{\frac{1}{2}} \right\|^4_{\text{BMO},\frac{1}{2}K_f}  .
	\end{split}
\end{equation}

\end{lemma}
\begin{proof}
We first define $\widetilde Y^{m,i}_t=e^{-K t}Y^{m,i}_t$ and $\widetilde Z^{m,i}_t=e^{-K  t}Z^{m,i}_t$, which satisfy the following system of BSDEs: 
	\begin{equation}\label{BSDE:general-k-P-tilde}
		\left\{\begin{aligned}
			-d\widetilde Y^{m,i}_t=&~\left[(K-\nu^{i}_t)\widetilde Y^{m,i}_t+\sum^\ell_{j=1}q^{ij}\widetilde Y^{m,j}_t       +\sum^\ell_{j=1}h_m(\kappa^{ij}_t)\widetilde Y^{m,j}_t \right.\\
			&~+h_m(g^{i}_t)^\top\widetilde Z^{m,i}_t+   e^{-K t}f^{i}_t\Bigg]\,dt -(\widetilde Z^{m,i}_t)^\top \,d\mathcal W_t,\quad t\in[0,m),\\
			\widetilde Y^{m,i}_t=&~0,\quad t\in[m,\infty),\\
			i\in&~\mathcal M.
		\end{aligned}\right.
\end{equation}
Due to the truncation, the system of BSDEs \eqref{BSDE:general-k-P-tilde} admits a Lipschitz driver. As a result, the well-posedness of \eqref{BSDE:general-k-P-tilde}—and hence of \eqref{BSDE:general-k-P}—in the space \( S^{2,0} \times M^{2,0} \) follows from a standard result; see \cite[Theorem 4.1]{Briand03}.

Next we verify that \( (Y^{m}, Z^{m}) \in \bigcap_{p\geq 1}S^{p, K}\times M^{2, K'} \). To this end, we introduce the following auxiliary system of BSDEs and compare its solution with that of \eqref{BSDE:general-k-P-tilde}:
	\begin{equation}\label{general-Y-bar}
		\left\{ \begin{aligned}
			-d\bar Y^{m,i}_t=&~\left[\sum^\ell_{j=1}q^{ij}\bar Y^{m,j}_t+| \kappa_t| \sum^\ell_{j=1} \bar Y^{m,j}_t+|g_t| |\bar Z^{m,i}_t|+e^{-K t}| f_t|\right]\,dt-(\bar Z^{m,i}_t)^\top\,d\mathcal W_t, \quad t\in[0,m),\\
			\bar Y^{m,i}_{t}=&~0,\quad t\in[m,\infty),\\
			 i\in&~\cal M,
		\end{aligned}\right.
	\end{equation}
	where we have used the following notations of Euclidean norm
	\[
		|\kappa|=\left(\sum^\ell\limits_{i,j=1}|\kappa^{ij}|^2\right)^\frac{1}{2},\quad |g|=\left(\sum^\ell\limits_{i=1}|g^{i}|^2\right)^\frac{1}{2},\quad |f|=\left(\sum^\ell\limits_{i=1}|f^{i}|^2\right)^\frac{1}{2}.
	\]

	{\bf Step 1:} {\it the estimate for $\bar Y^{m,i}$.}
	Indeed, any solution to the following one-dimensional BSDE is also a solution to the system \eqref{general-Y-bar}  
	\begin{equation}\label{general-Y-bar-2}
		\left\{ \begin{aligned}
			-d\bar Y^{m}_t=&~\left[\ell| \kappa_t| \bar Y^{m}_t+(\widetilde g_t^m)^\top \bar Z^{m}_t+e^{-K t}| f_t|\right]\,dt-(\bar Z^{m}_t)^\top\,d\mathcal W_t, \quad t\in[0,m),\\
			\bar Y^{m}_{t}=&~0, \quad t\in[m,\infty),
		\end{aligned}\right.
	\end{equation}
	where $\widetilde g_t^{m}:= |g_t|\frac{\bar Z^{m}_t}{|\bar Z^{m}_t|}\mathbf{1}_{\{|\bar Z^{m}_t|\neq 0\}}$ for each $t$. The explicit solution of \eqref{general-Y-bar-2} is given by (see e.g. \cite[Proposition 6.2.1]{Pham})
	\[
	\bar Y^m_t=(\Psi_t^m)^{-1}\mathbb E_t\left[\int^{m}_te^{-K s}\Psi_s^m |f_s|\,ds\right] \mathbf{1}_{\{ t\leq m\}}  \ge0,
	\]
	where $\Psi^m$ satisfies
	\begin{equation}\label{eq:Psi-k}
		\begin{split}
			\Psi^m_t=\mathcal E( \widetilde g^m )_t\Gamma_t, \quad 	\text{with } \Gamma_t:=\exp\left\{\int^t_0\ell  |\kappa_s|\,ds\right\}. 
		\end{split}
	\end{equation}
	By Assumption \ref{ass:general-BSDE}(iv), we have $| \kappa|\leq |\phi|^r$. Thus, for each $p\ge 1$, applying Lemma \ref{lemma:zeta-BMO} with $b_1=r$, $b_2=p\ell$, $b_3=0$ and $K_\zeta=K_\phi<0$, we obtain
	\begin{equation}\label{inequ:Gamma-k}
		\sup_\tau\mathbb E_\tau\left[\left|\frac{\Gamma_\infty}{\Gamma_\tau}\right|^p\right]\leq \sup_\tau\mathbb E_\tau\left[ e^{ p\ell \int^\infty_\tau|\phi_s|^rds}\right]\leq \left\| |\phi|^r \right\|_{p\ell,exp}<\infty.
	\end{equation}
	Since $|\widetilde g^m|\leq| g|$, it follows that $ \left\|  \widetilde g^m \right\|_{\text{BMO},0}<\infty$. It implies that 
	$$
	N^m:=\left\| \widetilde g^m \right\|_{\text{BMO},0}\leq   \left\| g \right\|_{\text{BMO},0}:=\bar N.
	$$
	By \cite[Theorem 3.1]{Kazamaki-1994}, $\mathcal{E}(\widetilde g^m)$ satisfies the reverse H\"older inequality. Using Lemma \ref{lemma:reverse},  for each $1 < q < q_{*}:=\Phi^{-1}(\bar N)\leq \Phi^{-1}(N^m)$ and each $\tau \in\mathcal T$, it holds 
	\begin{equation}\label{reverse}
		\mathbb{E}_\tau \left[ \left|\frac{\mathcal{E}(\widetilde g^m)_{\infty} }{\mathcal{E}(\widetilde g^m)_\tau } \right|^q\right] \leq \mathcal K(q, N^m) \leq \mathcal K(q, \bar N).
	\end{equation}
	Let $p_*$ be the conjugate exponent of $q_*$, and fix any $a>p_*$ with the corresponding conjugate exponent $q$. Then $q<q_*
	$ and 
	\begin{equation}\label{estimate:bar-Y-k}
		\begin{split}
			\bar Y^m_t=&~(\Psi_t^m)^{-1}\mathbb E_t\left[\int^{m}_te^{-K s}\Psi^m_s|f_s|\,ds\right]\\
			\leq&~\mathbb E_t\left[\sup_{u\ge t}\left|\frac{\Psi^m_u}{\Psi^m_t}\right|\int^{\infty}_te^{-K s}| f_s|\,ds\right]\\
			=&~ \mathbb E_t\left[ \sup_{u\geq t} \frac{\mathcal E(\widetilde g^m)_u}{\mathcal E( \widetilde g^m)_t}   \frac{ \exp(\int_0^u\ell| \kappa_s |\,ds) }{ \exp( \int_0^t\ell |\kappa_s|\,ds       )    }  \int_t^\infty e^{-K s}|f_s|\,ds  \right]  \qquad(\text{by \eqref{eq:Psi-k}}) \\
			\leq&~ \mathbb E_t\left[  \sup_{u\geq t} \left|\frac{\mathcal E(\widetilde g^m)_u}{\mathcal E( \widetilde g^m)_t}\right|^q \right]^{\frac{1}{q}} \mathbb E_t\left[  \left|\frac{ \Gamma_\infty }{ \Gamma_t   }\right|^a  \left(\int_t^\infty e^{-K s}|f_s|\,ds\right)^a    \right]^{\frac{1}{a}}          \qquad(\text{by H\"older's inequality})  \\
			\leq&~\frac{q}{q-1}\mathbb E_t\left[\left|\frac{ \mathcal{E}(\widetilde g^m)_{\infty}}{ \mathcal{E}(\widetilde g^m)_t}\right|^q\right]^\frac{1}{q}\mathbb E_t\left[\left|\frac{\Gamma_\infty}{\Gamma_t}\right|^a\left(\int^\infty_te^{-K s} |f_s|\,ds\right)^a\right]^\frac{1}{a}\quad(\text{by conditional Doob's maximal inequality})\\
			\leq&~  \frac{q}{q-1}\mathcal K(q, \bar N)^\frac{1}{q}\mathbb E_t\left[\left|\frac{\Gamma_\infty}{\Gamma_t}\right|^{2a}\right]^\frac{1}{2a}\mathbb E_t\left[\left(\int^\infty_te^{-K s} |f_s|\,ds\right)^{2a}\right]^\frac{1}{2a} \quad(\text{by H\"older's inequality and \eqref{reverse}}) \\
		\leq&~		\frac{q}{q-1} \mathcal K(q,\bar N)^\frac{1}{q} \left\| |\phi|^r\right\|^\frac{1}{2a}_{2a\ell,exp} \mathbb E_t\left[ \left(  \int_0^\infty e^{-K_fs} |	f_s	| \,ds  \right)^{2a} \right]^{\frac{1}{2a}}		\quad(\text{by }\eqref{inequ:Gamma-k} \text{ and } K >K_f ).
		\end{split}
	\end{equation}
Applying Doob's maximal inequality, for any $p>2a$, we obtain
\begin{equation*}
	\begin{split}
		\mathbb E\left[  \sup_{t\geq 0}(\bar Y^m_t)^p  \right] \leq&~  \left(\frac{q}{q-1}\right)^p \mathcal K(q,\bar N)^\frac{p}{q} \left\| |\phi|^r\right\|^\frac{p}{2a}_{2a\ell,exp} \mathbb E\left[   \sup_{t\geq 0}  \mathbb E_t\left[ \left(  \int_0^\infty e^{-K_fs} |	f_s	| \,ds  \right)^{2a} \right]^{\frac{p}{2a}}	 \right]\\
		\leq&~ \left(\frac{q}{q-1}\right)^p  \left( \frac{p}{p-2a}   \right)^{\frac{p}{2a}} \mathcal K(q,\bar N)^\frac{p}{q} \left\| |\phi|^r\right\|^\frac{p}{2a}_{2a\ell,exp} \mathbb E\left[ \left( 	\int_0^\infty  e^{-K_fs} | f_s |\,ds	  \right)^p \right].
	\end{split}
\end{equation*}


{\bf Step 2:} {\it the estimate for $\widetilde Y^{m,i}$ and $Y^{m,i}$.}
The system of BSDEs \eqref{BSDE:general-k-P-tilde} satisfies all the conditions of \cite[Lemma 2.2]{Hu-Liang-Tang} (see also \cite{Hu-Peng-2006}). In particular, the inequality \( (K - \nu^{i}_t)\bar{Y}^m_t \leq 0 \) holds under the assumption \( K < \bar{K} \) and Assumption~\ref{ass:general-BSDE}(v). This implies that 
	\[ 
		(K-\nu^{i}_t)\bar Y^{m}_t+\sum^\ell_{j=1}q^{ij}\bar Y^{m}_t+\sum^\ell_{j=1}h_m(\kappa^{ij}_t)\bar Y^{m}_t\leq \ell| \kappa_t| \bar Y^{m}_t.
	\] 
Moreover, by the definitions of $\widetilde g^m$ and $h_m$, we have
	\[
		h_m(g^{i}_t)^\top \bar Z^{m}_t\leq (\widetilde g_t^m)^\top \bar Z^{m}_t \text{ for }i\in\cal M.
	\]
	Thus, Condition (iv) in \cite{Hu-Liang-Tang} is satisfied. The driver of \eqref{BSDE:general-k-P-tilde} satisfies Condition (iii) in \cite{Hu-Liang-Tang}, because $q^{ij}\geq 0$ and $\kappa^{ij}\geq 0$. Conditions (i) and (ii) are satisfied trivially.
Therefore, by applying \cite[Lemma 2.2]{Hu-Liang-Tang}, we conclude that for each $i=1,\cdots,\ell$,
\begin{equation}\label{compare}
			|\widetilde  Y^{m,i}_t| \leq \bar Y_t.
\end{equation}
As a result, the estimate \eqref{estimate:Y-Pik-2} follows. 
	
	{\bf Step 3: } {\it the estimate for $Z^{m,i}$.}  
	Applying It\^o's formula to $e^{-2K'\cdot}|Y^{m,i}_\cdot|^2$, we have for each $\tau\in\cal T$,
	\begin{equation*}
		\begin{aligned}
			&~e^{-2K' \tau }|Y^{m,i}_\tau|^2+ \mathbb E_\tau\left[\int^{m}_\tau e^{-2K' s}|Z^{m,i}_s|^2\,ds\right]\\
			\leq&~ \mathbb E_\tau\left[\int^{m}_\tau 2e^{-2K' s}Y^{m,i}_s\left(h_m(g^{i}_s)^\top Z^{m,i}_s+(K'-\nu^{i}_s)Y^{m,i}_s+\sum^\ell_{j=1}q^{ij}Y^{m,j}_s+\sum^\ell_{j=1}h_m(\kappa^{ij}_s)Y^{m,j}_s+f^{i}_s\right)\,ds\right]\\
			\leq&~\frac{1}{4}\mathbb E_\tau\left[\int^{m}_\tau e^{-2K' s}|Z^{m,i}_s|^2\,ds\right]+4\mathbb E_\tau\left[\int^{m}_\tau e^{-2K' s}|Y^{m,i}_s|^2|h_m(g^{i}_s)|^2\,ds\right]\\
			&~+\mathbb E_\tau\left[\int^{m}_\tau 2e^{-2K' s}Y^{m,i}_s\left((K'-\bar K)Y^{m,i}_s+\sum^\ell_{j=1}q^{ij}Y^{m,j}_s+\sum^\ell_{j=1}h_m(\kappa^{ij}_s)Y^{m,j}_s+f^{i}_s\right)\,ds\right],
		\end{aligned}
	\end{equation*}
	which implies that  
	\begin{equation*}
		\begin{split}
			&~\mathbb E_\tau\left[\int^{m}_\tau e^{-2K' s}|Z^{m,i}_s|^2\,ds\right]\\
			\leq&~ 4(K'-\bar K)\mathbb E_\tau\left[ \int_\tau^{m}  e^{-2K' s} (Y^{m,i}_s)^2\,ds  \right] +8\mathbb E_\tau\left[ \int_\tau^{m} e^{-2K' s} |h_m(g^{i}_s)|^2(Y^{m,i}_s)^2\,ds \right]  \\
			&~+4\mathbb E_\tau\left[\int^{m}_\tau e^{-2K' s}Y^{m,i}_s\left(\sum^\ell_{j=1}q^{ij}Y^{m,j}_s \right)\,ds\right] +4\mathbb E_\tau\left[\int^{m}_\tau e^{-2K' s}Y^{m,i}_s\left( \sum^\ell_{j=1}h_m(\kappa^{ij}_s)Y^{m,j}_s\right)\,ds \right]\\
			&~+4\mathbb E_\tau\left[\int^{m}_\tau e^{-2K' s}Y^{m,i}_s  f^{i}_s \,ds\right].
		\end{split}
	\end{equation*}
By the assumption that $K'>K>K_f$ and the fact that $|Z^{m,i}_t|=0$ for $t\in (m,\infty)$, we get that
\begin{equation}\label{estimate:Z-Pik-1}
	\begin{split}
		\mathbb E_{\tau}\left[\int^{\infty}_{\tau} e^{-2K' s}| Z^{m}_s|^2\,ds\right]
		\leq &~c\mathbb E_{\tau}\left[  \sup_{t\geq {\tau}} e^{-2K  t}| Y_t^{m}  |^{2}   \right]		+ c\mathbb E_{\tau}\left[  \sup_{t\geq {\tau}}  e^{-4K t} | Y_t^{m} |^{4}    \right]^{1/2}	 \mathbb E_{\tau}\left[ \left( \int_{\tau}^\infty |g_s|^2\,ds  \right)^{2} \right]^{1/2} 			\\
		&~  +	c \mathbb E_{\tau}\left[  \sup_{t\geq {\tau}} e^{-4K t}| Y^{m}_t |^{4} \right]^{1/2}   \mathbb E_\tau\left[ \left( \int_{\tau}^\infty |\phi_s|^2\,ds  \right)^{r} \right]^{1/2} 		\\
		&~+ c\mathbb E_{\tau}\left[  \sup_{t\geq {\tau}}  e^{-2K t}  |Y_t^{m}  |^{2}  \right]^{1/2}\mathbb E_{\tau}\left[     \left( \int_{\tau}^\infty  e^{-K_{f}s} |f_s| \,ds \right)^{2}      \right]^{1/2},
	\end{split}
\end{equation}
from which we obtain the estimate \eqref{estimate:Z-Pik-2} for $\tau=0$.

  If we further assume that $|f^{}|^{1/2}\in H^{2, K_f/2}_{\text{BMO}}$, then by \eqref{estimate:bar-Y-k}-\eqref{estimate:Z-Pik-1} together with Lemma \ref{lemma:energy}, we can easily obtain the desired estimates \eqref{estimate:Y-Pik} and \eqref{estimate:Z-Pik}.
\end{proof}  
\begin{remark}\label{rem:additional-estimate}
 By a similar argument as in Step 3, we can also obtain the following estimate for $Z^{m}$ for each $p\geq 1$
\begin{equation}\label{estimate:Z-Pik-3}
	\begin{split}
	&~\mathbb E\left[\left(\int^{\infty}_0 e^{-2K' s}| Z^{m}_s|^2\,ds\right)^p\right]\\
	\leq &~c\mathbb E\left[  \sup_{t\geq 0} e^{-2pK  t}| Y_t^{m}  |^{2p}   \right]		+ c\mathbb E\left[  \sup_{t\geq 0}  e^{-4pK t}  |Y_t^{m} |^{4p}    \right]^{1/2}	\mathbb E\left[ \left( \int_0^\infty |g_s|^2\,ds  \right)^{2p} \right]^{1/2} 		\\
	&~  +	c \mathbb E\left[  \sup_{t\geq 0} e^{-4pK t} |Y^{m}_t |^{4p} \right]^{1/2}   \mathbb E\left[ \left( \int_0^\infty |\phi_s|^2\,ds  \right)^{rp} \right]^{1/2} 		\\
	&~+ c\mathbb E\left[  \sup_{t\geq 0}  e^{-2pK t} | Y_t^{m}  |^{2p}  \right]^{1/2}\mathbb E\left[     \left( \int_0^\infty  e^{-K_{f}s} |f_s| \,ds \right)^{2p}      \right]^{1/2}.
	\end{split}
\end{equation}
This estimate will be used in the proof of Theorem \ref{thm:general-YZ}.
\end{remark}
%
%
\begin{theorem}\label{thm:general-YZ}
	 For any given $K_Y$ and $K_Z$ satisfying $K_f<K_Y<K_Z<\bar K$, the system of BSDEs \eqref{BSDE:Y-infinite} admits a solution $(Y,Z)$ in $\bigcap_{p\geq 1}S^{p, K_Y}\times M^{2, K_Z}$, satisfying the following transversality condition 
	\begin{equation}\label{trans:general-BSDE}
		\begin{split}
			\lim_{T\rightarrow\infty}\mathbb E\left[    e^{-pK_YT} |Y_T|^p   \right]=0.
		\end{split}
	\end{equation} 
If we further assume that $|f^{}|^{1/2}\in H^{2, K_f/2}_{\text{BMO}}$, the solution $(Y,Z)$ belongs to a finer space $L^{\infty,K_Y} \times H^{2,K_Z}_{\text{BMO}}$.
\end{theorem}
\begin{proof}
	Given $K_Y$ and $K_Z$ as in the statement of the theorem, choose constants $K_F$, $K$ and $K'$ such that
	\[
		K_f<K<K'<K_F<K_Y<\bar K.
	\]
For any \( m_1, m_2 \in \mathbb{N} \), denote by \( (Y^{m_2}, Z^{m_2}) \) and \( (Y^{m_1}, Z^{m_1})\in L^{\infty,K}\times H^{2,K'}_{\text{BMO}} \) the solutions to the BSDE \eqref{BSDE:general-k-P} with \( m = m_2 \) and \( m = m_1 \), respectively, as constructed in Lemma~\ref{lemma:general-P}. Define the difference processes
\[
	\left( \Delta Y(m_2, m_1),\ \Delta Z(m_2, m_1) \right) := \left( Y^{m_2} - Y^{m_1},\ Z^{m_2} - Z^{m_1} \right).
\]
Then \( \left( \Delta Y(m_2, m_1),\ \Delta Z(m_2, m_1) \right) \) satisfies the BSDE system \eqref{BSDE:general-k-P} with \( m = m_2 \) and $f^{i}$ replaced by 
\[
	 F^{i}_s(m_2,m_1):=	\left(  h_{m_2}(g^{i}_s) - h_{m_1}( g^{i}_s )  \right)^\top Z^{m_2,i}_s + \sum_{j=1}^\ell \left(     	h_{m_2}(\kappa^{ij}_s) - h_{m_1}(\kappa^{ij}_s)		  \right) Y^{m_2,j}_s+f^{i}_s\textbf{1}_{[m_1,m_2)}.
\]
{\bf Step 1:} {\it the estimate and convergence of $F^{i}$.} In this step, we verify $F^{i}$ satisfies Assumption \ref{ass:general-BSDE}(ii) for the constant $K_F$.

By H\"older's inequality we have the following convergence estimate for each $p\ge 1$:
\begin{equation}\label{convergence:F-m}
	\begin{split}
		&~\mathbb E\left[ \left(\int_0^\infty e^{-K_F s}|  F^{i}_s(m_2,m_1)|\,ds\right)^p\right]\\
		\leq&~ c\sup_{i\in\cal M} \mathbb E\left[\left(\int^\infty_0 |h_{m_2}(g^{i}_s)-h_{m_1}(g^{i}_s)|^2\,ds\right)^p\right]^\frac{1}{2}\mathbb 
		E\left[\left(\int^\infty_0 e^{-2K's}|Z^{m_2,i}_s|^2\,ds\right)^p\right]^\frac{1}{2}\\
		&~ +c \mathbb E\left[  \sup_{t\geq 0}  e^{-2pK t} | Y_t^{m}  |^{2p}  \right]^{1/2} \mathbb E\left[\left(\int^\infty_0   e^{-(K_F-K)s}   |h_{m_2}(\kappa^{ij}_s)-h_{m_1}(\kappa^{ij}_s)|\,ds\right)^{2p}\right]^{\frac{1}{2}}\\  
		&~ +c\mathbb E\left[\left(\int^\infty_0  e^{-K_fs}  |f^{i}_s\textbf{1}_{[m_1,m_2)}|\,ds\right)^p\right]\\
		{\rightarrow}&~   0,\qquad \text{ as }{m_2, m_1 \to \infty}.
	\end{split}
\end{equation}

This convergence follows from the following two facts:
\begin{itemize}
 \item[i).]    The uniform bounds hold:    
 \[
 		\sup_m\|Y^{m}\|_{2p,K}<\infty, \quad \text{and} \sup_m\left\|  Z^{m,i}  \right\|_{M^{2p,K'}}<\infty
 \]
  by estimates \eqref{estimate:Y-Pik-2} and \eqref{estimate:Z-Pik-3}. 

 \item[ii).]  The truncation error vanishes as \( m \to \infty \): for example,
 \begin{equation*}
 	\begin{split}
 		  &~\mathbb E\left[\left(\int^\infty_0 |h_{m}(g^{i}_s)- g^{i}_s|^2\,ds\right)^p\right] \\
 		  =&~ \mathbb E\left[\left(\int^\infty_0 \left| \frac{m}{|g^{i}_s|} g^{i}_s \mathbf{1}_{\{ |g^{i}_s|\geq m \}} +  g^{i}_s\mathbf{1}_{\{0< |g^{i}_s|< m \}} -  g^{i}_s \right|^2\,ds\right)^p\right]\\
 		 \leq&~ 2^{2p+1}\mathbb E\left[   \left(\int_0^\infty | g^{i}_s |^2  \mathbf{1}_{\{ |g^{i}_s|\geq m \}} \,ds\right)^p   \right]\\
 		  \rightarrow&~ 0,
 	\end{split}
 \end{equation*}
by Assumption \ref{ass:general-BSDE}(ii). A similar argument applies to $h_m(\kappa), f\textbf{1}_{[0,m)}$.
\end{itemize}
 {\bf Step 2:} {\it convergence of $\Delta Y^{i}$ and $\Delta Z^{i}$.} By Lemma \ref{lemma:general-P}, $( \Delta Y,\Delta Z )\in \bigcap_{p\geq 1}S^{p, K_Y}\times M^{2, K_Z}$. In this step, we consider the convergence of the difference process.

By \eqref{estimate:Y-Pik-2} in Lemma \ref{lemma:general-P} and the convergence \eqref{convergence:F-m}, we obtain for any $p>2a$
\begin{equation}\label{convergence:Y-m}
	\begin{split}
		&~ \mathbb E\left[ \sup_{t\geq 0}  \left| e^{-K_Yt} \Delta  Y^{i}_t(m_2,m_1) \right|^p  \right] \\ 
	\leq&~ 			   \left(\frac{q}{q-1}\right)^p  \left( \frac{p}{p-2a}   \right)^{\frac{p}{2a}}\mathcal K(q,\bar N)^\frac{p}{q} \left\| |\phi_s|^r\right\|^\frac{p}{2a}_{2a\ell,exp} \mathbb E\left[ \left( 	\int_0^\infty  e^{-K_Fs} | F_s(m_2,m_1) |\,ds	  \right)^p \right] \\  
	\rightarrow&~ 0,\quad \text{as }m_2 \text{ and }m_1\rightarrow\infty. 
	\end{split}
\end{equation}
By the estimate \eqref{estimate:Z-Pik-2}, we have
\begin{equation}\label{convergence:Z-m}
	\begin{split}
		&~\mathbb E\left[\int_0^\infty e^{-2K_Zs}| \Delta Z_s(m_2,m_1)|^2\,ds\right]\\
		\leq&~	c\mathbb E\left[  \sup_{t\geq 0} e^{-2K_Yt}| \Delta Y_t(m_2,m_1) |^2   \right]		+ c\mathbb E\left[  \sup_{t\geq 0} | e^{-K_Yt}\Delta Y_t(m_2,m_1)   |^4    \right]^{1/2}\mathbb E\left[ \left( \int_0^\infty |g_s|^2\,ds  \right)^{2} \right]^{1/2} 			\\
		&~+	c \mathbb E\left[  \sup_{t\geq 0} |e^{-K_Yt}\Delta Y_t(m_2,m_1)|^4 \right]^{1/2}  \mathbb E\left[ \left( \int_0^\infty |\phi_s|^2\,ds  \right)^{r} \right]^{1/2}	\\
		&~+ c\mathbb E\left[  \sup_{t\geq 0} | e^{-K_Yt}\Delta Y_t(m_2,m_1) |^2  \right]^{1/2}\mathbb E\left[     \left( \int_0^\infty  e^{-K_Fs}|F_s|\,ds \right)^2      \right]^{1/2} \\
		&\xrightarrow{m_2, m_1 \to \infty} 0.
	\end{split}
\end{equation}
%
Thus, the sequence $(Y^{m},Z^{m})$ is a Cauchy one in $\bigcap_{p\geq 1}S^{p,K_Y}\times M^{2,K_Z}$, and admits a limit denoted by $(Y,Z)\in \bigcap_{p\geq 1}S^{p,K_Y}\times M^{2,K_Z}$.  

{\bf Step 3:} {\it the limit $(Y,Z)$ is a solution to the system \eqref{BSDE:Y-infinite}.}  Our goal is to prove that each term in the BSDE for $(e^{-K_Z\cdot} Y^{m}_\cdot,e^{-K_Z\cdot}Z^{m}_\cdot)$ converges a.s. to the corresponding term in the BSDE for $(e^{-K_Z\cdot} Y_\cdot,e^{-K_Z\cdot}Z_\cdot)$, by extracting a subsequence if necessary. Note that $(e^{-K_Z\cdot} Y^{m}_\cdot,e^{-K_Z\cdot}Z^{m}_\cdot)$ satisfy the BSDE \eqref{BSDE:general-k-P-tilde} with $K$ replaced by $K_Z$.

First, by the convergence of $Y^{m}$ in Step 2, we have
  \begin{equation*}
  	\begin{split}
  		&~\lim_{m\rightarrow\infty }  \int_0^\infty \left(| K_Z-\nu^{i}_s | e^{-K_Zs} | Y^{m,i}_s - Y^{i}_s	|	+ \sum_{j=1}^\ell q^{ij} e^{-K_Z s} |	Y^{m,j}_s - Y^{j}_s		|		\right)	\,ds  \\
  		\leq&~ c\lim_{m\rightarrow\infty }  \sup_{t\geq 0} e^{-K_Y t}| Y^{m}_t - Y_t   |  \\
  		=&~ 0,
  	\end{split}
  \end{equation*}
  where we have used the fact $K_Y<K_Z$.
  
  Second, using again the convergence of $Y^{m}$ in Step 2 and the intergrability of $\kappa$ in Assumption \ref{ass:general-BSDE}(iv), we obtain
  \begin{equation*}
  	\begin{split}
  		 &~\lim\limits_{m\rightarrow\infty }   \int_0^{\infty} e^{-K_Zs}\sum^\ell_{j=1}|h_{m}(\kappa^{ij}_s)Y^{m,j}_s- \kappa^{ij}_sY^{j}_s|\,ds  \\
  		 \leq&~c\lim\limits_{m\rightarrow\infty }   \int_0^{\infty} e^{-K_Zs} \max_{j\in\cal M}|h_{m}(\kappa^{ij}_s)||Y^{m}_s-  Y_s|\,ds   + c\lim\limits_{m\rightarrow\infty }   \int_0^{\infty} e^{-K_Zs} \max_{j\in\cal M}|h_{m}(\kappa^{ij}_s)-\kappa_s^{ij}| | Y_s|\,ds  \\
  		 \leq&~ \lim_{m\rightarrow\infty} \sup_{t\geq 0} e^{-K_Yt}| Y^{m}_t - Y_t  |   \int_0^\infty  e^{-(K_Z-K_Y)s} |\kappa_s|\,ds + c\sup_{t\geq 0} e^{-K_Yt}|Y_t|\lim_{m\rightarrow\infty} \int_0^\infty e^{-(K_Z-K_Y)s} | \kappa_s| \mathbf{1}_{\{ | \kappa_s |>m \}}         \,ds \\
  		 =&~0,
  	\end{split}
  	\end{equation*}
    where we have used again the fact $K_Y<K_Z$.
    
 Similarly, we get 
\begin{equation*}
 	\begin{split}
 		&~ \lim\limits_{m\rightarrow\infty } \int_0^{\infty} e^{-K_Zs}|h_{m}(g^{i}_s)^\top Z^{m,i}_s- (g^{i}_s)^\top Z^{i}_s|\,ds \\
 		\leq&~ c\lim\limits_{m\rightarrow\infty }  \left(  \int^\infty_0 |h_{m}(g^{i}_s)-g^{i}_s|^2\,ds\right)^\frac{1}{2}\left(   \int^\infty_0 e^{-2K_Zs}|Z_s|^2\,ds\right)^\frac{1}{2}+c\lim\limits_{m\rightarrow\infty } \left(\int^\infty_0 |g_s|^2\,ds\right)^\frac{1}{2}\left( \int^\infty_0 e^{-2K_Zs}|Z^{m}_s-Z_s|^2\,ds\right)^\frac{1}{2}\\
 		= &~ 0,
 	\end{split}
 \end{equation*}
and
\begin{equation*}
	\begin{split}
		~ \lim\limits_{m\rightarrow\infty } \int_0^{\infty} e^{-K_Zs}|f^{i}_s\mathbf{1}_{[0,m)}(s)-f^{i}_s|\,ds 
		\leq~ c\lim\limits_{m\rightarrow\infty } \int_0^{\infty} e^{-K_Zs}|f^{i}_s|\mathbf{1}_{[m,\infty)}\,ds
		= ~ 0. 
	\end{split}
\end{equation*}

Finally, for any $t<T$ it holds 
  \begin{equation*}
  	\begin{split}
  	\lim\limits_{m\to\infty}\mathbb E\left[\left|\int^{T}_te^{-K_Zs}(Z^{m,i}_s-Z^{i}_s)^\top \,d\mathcal W_s\right|^2\right]\leq c\lim_{m\rightarrow\infty} \mathbb E\left[\int^\infty_0e^{-2K_Z s}|Z^{m,i}_s-Z^{i}_s|^2ds\right]= 0.
  	\end{split}
  \end{equation*}
{\bf Step 4:} {\it the limit $(Y^{i},Z^{i})$ resides in a finer space if $|f^{}|^{1/2}\in H^{2, K_f/2}_{\text{BMO}}$.}
By the convergence \eqref{convergence:Y-m} and the estimate \eqref{estimate:Y-Pik}, and by noting $K_Y>K$, we obtain the following almost sure convergence (up to a subsequence)
\[
e^{-K_Yt}|Y^{i}_t|	=  \lim_{m\rightarrow\infty}e^{-K_Y t}| Y^{m,i}_t 	|\leq 	\ \frac{q}{q-1} \mathcal K(q,\bar N)^\frac{1}{q}\left\| |\phi|^r\right\|^\frac{1}{2a}_{2a\ell,exp} \left(\lceil 2a\rceil !  \right)^{\frac{1}{2a}}  \left\|	  |f|^{\frac{1}{2}}	\right\|^{2}_{\text{BMO},\frac{1}{2}K_f}<\infty,
\]
which implies $Y^{i}\in L^{\infty,K_Y} $. By the convergence \eqref{convergence:Z-m} and the estimate \eqref{estimate:Z-Pik}, and by noting $K_Z>K'$ and $K_Y>K$ we have the following almost sure convergence (up to a subsequence)
\begin{equation*} 
	\begin{split}
		&~\mathbb E_\tau\left[\int^{\infty}_\tau e^{-2K_{Z}s}|Z_s|^2\,ds\right]=\lim_{m\rightarrow\infty }\mathbb E_\tau\left[\int^{\infty}_\tau e^{-2K_{Z}s}|Z^{m}_s|^2\,ds\right]\\
		\leq&~c\sup_{m}\| Y^{m} \|^2_{\infty,K_Y} + c\sup_{m}\|Y^{m}\|^2_{\infty,K_Y} \left(\|g\|^2_{\text{BMO}}+ \left\|  \phi \right\|^{r}_{\text{BMO},K_\phi}\right)   + c \left\|  |  f  |^{\frac{1}{2}} \right\|^4_{\text{BMO},\frac{1}{2}K_f}\\
		<&~\infty,
	\end{split}
\end{equation*}
where the above inequality is obtained by \eqref{estimate:Y-Pik}. Thus, $Z^{i}\in H^{2,K_Z}_{\text{BMO}}$. 

{\bf Step 5.} For each $p\geq 1$, the following convergence implies the transersality condition \eqref{trans:general-BSDE}:
\begin{equation*}
	\begin{split}
	&~\sup_{T\geq m} \mathbb E\left[  e^{-p K_Y T} |Y_T|^p    \right] \\
	\leq&~ c\mathbb E\left[ \sup_{t\geq 0}e^{-pK_Y t} | Y^m_t-Y_t	|^p     \right]  +c \sup_{T\geq m} \mathbb E\left[ e^{-pK_Y T} | Y^m_T |^p  \right]\\
	=&~c\mathbb E\left[ \sup_{t\geq 0}e^{-pK_Y t} | Y^m_t-Y_t	|^p     \right] \\
	\rightarrow&~ 0,\quad \text{as }m\rightarrow\infty,
	\end{split}
\end{equation*}
by the convergence in Step 2.
\end{proof}

\subsection{The BSDE Systems \eqref{BSDE:A-inf}-\eqref{BSDE:C-inf}: Existence}\label{sec:convergence-BC}

In this section, we establish the existence of solutions to the infinite-horizon BSDE systems \eqref{BSDE:A-inf}–\eqref{BSDE:C-inf}. Among these, the system for \( A \), given by \eqref{BSDE:A-inf}–\eqref{trans:A}, can be treated by adapting the argument in \cite[Theorem 4.2]{Hu-COCV}; the relevant results are summarized in Appendix \ref{app:A}. The existence results for the systems corresponding to \( B \) and \( C \) will be established using the framework developed in Section~\ref{subsec:general-BSDE}.

The following estimate for $\frac{1}{\gamma}$ will be used frequently.
\begin{lemma}\label{lemma:gamma}
	Under Assumption \ref{ass:standing}, for each $p\ge 1$, we have  
	$$\mathbb E\left[\sup_{t\ge 0}\left|\frac{1}{\gamma_t}\right|^p\right]<\infty.$$
\end{lemma}
\begin{proof}
	From \eqref{eq:gamma} we have
	\begin{equation*}
		d\frac{1}{\gamma_s}=\frac{1}{\gamma_s}\Bigg(\left(-\mu_s+\sigma_{1,s}^2\right)\,ds-\sigma_{1,s}\,dW_s\Bigg),
	\end{equation*}
	which implies the closed form expression
	\begin{equation*}
		\frac{1}{\gamma_t}=\frac{1}{\gamma_0}\exp\left\{\int^t_0\left(-\mu_s+\frac{\sigma_{1,s}^2}{2}\right)\,ds-\int^t_0\sigma_{1,s}\,dW_s\right\}.
	\end{equation*}
	Since $|\mu_{t}|\leq \hat c e^{-Lt}$, and  $|\sigma_{1,t}|\leq \hat c e^{-Lt}$, we get the desired estimate.
\end{proof}

We first rewrite \eqref{BSDE:B-inf} in the form compatible with the framework of Section \ref{subsec:general-BSDE}. Specifically, we consider the following formulation:
\begin{equation}\label{BSDE:B-inf-2}
	\left\{
	\begin{aligned}
		&- d \bar{B}_t^{i} =\left\{ -\nu^{i}_t 
		\bar{B}_t^{i}  +\kappa^{i}_t\bar B^{i}_t + \sum_{i = 1}^\ell q^{ij} \bar{B}_t^{j} +  g^{i}_tZ^{\bar B,i}_t  +
		f^{\bar B,i}_t \right\} d t -Z_t^{\bar{B}, i} d W_t, \\
		&\lim\limits_{T\to\infty} \mathbb E\left[   e^{\frac{1}{4}\beta p T}|\bar B^i_T|^p \right]= 0 \quad 	\text{ for each }  p\geq 1 \text{ and }  i\in\mathcal  M.
	\end{aligned}\right.
\end{equation}
Here, the coefficients are given by
	\begin{align*}  
			\nu^{i}_t=&~  \frac{\phi}{2} +
			\frac{\gamma_t}{a^{i}_t}\left(\rho^i_t   +\mu_t - (\sigma_{1,t})^2    		  \right)\left( (\mu_t+\rho^i_t) \bar{A}^{i}_t + \frac{\lambda^i_t}{\gamma_t}  \right),			\qquad   \kappa^{i}_t = -
			\frac{ \gamma_t }{a^{i}_t } \left(  \rho^i_t  +\mu_t-(\sigma_{1,t})^2	\right)	\sigma_{1,t}
			Z^{\bar{A}, i}_t,\\
			f^{\bar B,i}_t =&~  \frac{\gamma_t}{a^{i}_t}\left( 2   \sigma_{1, t} \widetilde\sigma_t^i
			\bar{A}_t^{i} -   \sigma_{1, t} \widetilde\sigma_t^i     \right)
			\left( \mu_t \bar{A}_t^{i} + \sigma_{1, t} Z_t^{\bar{A}, i}
			+ \frac{\lambda_t^i}{\gamma_t} + \rho_t^i \bar{A}_t^{i}
			\right)  + 2    \widetilde\sigma_t^i Z_t^{\bar{A}, i},\\
			g^{i}_t=&~-\frac{\gamma_t}{a^{i}_t}   \sigma_{1,t}  
			\left( (\mu_t+\rho_t^i ) \bar{A}_t^{i} + \sigma_{1, t} Z_t^{\bar{A}, i}
			+ \frac{\lambda_t^i}{\gamma_t} \right).
	\end{align*}
 The next lemma confirms that the coefficients appearing in \eqref{BSDE:B-inf-2} fulfill all the conditions outlined in Assumption \ref{ass:general-BSDE} in Section~\ref{sec:general-BSDE}.
\begin{lemma}\label{lemma:nu-kappa-f-B}
	The processes $\nu$, $\kappa$, $f^{\bar B}$ and $g$ satisfy the requirements in Assumption \ref{ass:general-BSDE} in Section \ref{sec:general-BSDE}. In particular, $-\frac{\beta}{2}<K_{f^{\bar B}}<\frac{\phi}{2}$, where we recall $\beta$ is the positive constant appearing in Assumption \ref{ass:standing}.
\end{lemma}
\begin{proof}
	We verify each condition in Assumption \ref{ass:general-BSDE}.
	
	(i) Choose $r=1$, $\bar K=\frac{\phi}{2}$.
	
	(ii) 	Now let \( K_{f^{\bar{B}}} \in \left( -\frac{\beta}{2}, \frac{\phi}{2} \right) \). Then for each \( \tau \in \mathcal{T} \), we estimate:
	\begin{equation*}
		\begin{split}
			\mathbb E_\tau\left[\int^\infty_\tau e^{-  K_{f^{\bar B}}s}|f^{\bar B,i}_{s}| \,ds\right]\leq&~ c\mathbb E_\tau\left[\int^\infty_0e^{- (\frac{\beta}{2}+K_{f^{\bar B}})s}\,ds\right]+c\mathbb E_\tau\left[\int^\infty_\tau e^{-( \frac{\beta}{2}+ K_{f^{\bar B}})s}|Z_s^{\bar{A}}|\,ds\right]\\
			\leq&~ c+c  \| Z^{\bar A}\|_{\text{BMO},\frac{1}{4}\beta+\frac{1}{2}K_{f^{\bar B}}}\\
			<&~\infty,
		\end{split}
	\end{equation*}
	which implies $|f^{\bar B}|^{1/2}\in H^{2, \frac{1}{2}K_{f^{\bar B}}}_{\text{BMO}}$.

(iv) Let \( K_\phi \) be any constant satisfying \( -L < K_\phi < 0 \), where \( L \) is the constant appearing in Assumption~\ref{ass:standing}. 
By Assumption \ref{ass:standing}, $|\kappa|\leq c e^{-L\cdot} |Z^{\bar A}|:=|\phi|$, which belongs to $H^{2,K_\phi}_{\text{BMO}}$.


Similarly, by Proposition \ref{prop:A} and Corollary \ref{coro:gamma-a}, we also have (iii):
\[
\|g\|_{\text{BMO},0}<\infty.
\]		
(v) By Assumption \ref{ass:standing}, Proposition \ref{prop:A} and Corollary \eqref{coro:gamma-a}, $\bar K\leq \nu^{i}\leq c$ holds for some positive constant $c$.
\end{proof}
Let $K_{\bar B}=-\frac{\beta}{4}$, $K_{Z^{{\bar B}}}=-\frac{\beta}{8}$ and $K_{f^{\bar B}}=-\frac{\beta}{3}$. Given Lemma \ref{lemma:nu-kappa-f-B},  Theorem \ref{thm:general-YZ} then yields the next theorem which establishes the existence result for \eqref{BSDE:B-inf-2}.  
\begin{theorem}\label{thm:B}
	The system of infinite-horizon BSDEs \eqref{BSDE:B-inf-2} with transversality condition \eqref{trans:B} admits a solution  $(\bar B,Z^{\bar B})$ in $L^{\infty,-\frac{ \beta}{4}} \times H^{2,-\frac{\beta}{8}}_{\text{BMO}} $.
\end{theorem}


The system \eqref{BSDE:C-inf} is also a special case of \eqref{BSDE:Y-infinite} where all coefficients vanish, leaving only the following nonhomogeneous term:
\begin{equation*}
	\begin{split}
			f^{C,i}_t=&~- \frac{1}{4 a^{i}_t}\left( (\rho_t^i  +\mu_t-\sigma_{1,t}^2   ) \bar{B}_t^{i} + 2   \sigma_{1,t} \widetilde\sigma_t^i \bar{A}_t^{i} - \sigma_{1,t} \widetilde\sigma_t^i   +	\sigma_{1,t} Z^{\bar B,i}_t	  \right)^2  \\
		&~+ (\widetilde\sigma_{t}^i)^2  \left( \bar{A}_t^{i} - \frac{1}{2  } \right) \frac{1}{\gamma_t}  + \widetilde\sigma_t^i \left( Z_t^{\bar{B},i} -  \sigma_{1,t} \bar{B}_t^{i}  \right) \frac{1}{\gamma_t} .
	\end{split}
\end{equation*}
By Theorem \ref{thm:B}, Proposition \ref{prop:A}, Corollary \ref{coro:gamma-a} and Lemma \ref{lemma:gamma}, we have the following estimate for $f^{C}:$ 
	\begin{equation}\label{estimate:f-C}
		\begin{aligned}
			 \mathbb E \left[  \left(\int_0^\infty e^\frac{\beta s}{4}|f^{C}_s| \,ds \right)^p\right] < \infty, \quad p\geq 1.
		\end{aligned}
	\end{equation}
Thus all the coefficients in the system of $C$ satisfy the corresponding requirements in Assumption \ref{ass:general-BSDE} with $K_{f^{C}}=-\frac{\beta}{4}$. Setting $K_{C}=-\frac{\beta}{8}$, $K_{Z^C}=0$, we then obtain the existence result for \eqref{BSDE:C-inf} by applying Theorem \ref{thm:general-YZ}.  
\begin{theorem}\label{thm:C}
The infinite-horizon BSDE system \eqref{BSDE:C-inf}-\eqref{trans:C} admits a solution $(C,Z^{C})$ in $S^{2,-\frac{\beta}{8}} \times  M^{2,0}$.
\end{theorem}




\section{The solvability of the control problem \eqref{cost-inf-continuous}-\eqref{state-inf-continuous}}\label{sec:control}

In this section, we return to the infinite-horizon stochastic control problem \eqref{cost-inf-continuous}–\eqref{state-inf-continuous}. Our goal is to address the solvability of the control problem in two parts.

First, leveraging the BSDE results developed in Section~\ref{sec:BSDEs}, we prove that the value function of the control problem is finite and thus well-posed (see Definition~\ref{def:wellposed}). To this end, we begin with the corresponding finite-horizon stochastic control problem and reformulate the cost functional as the sum of a fixed cost and a completed square term. We then pass to the infinite-horizon limit.

Second, under an additional assumption on the BSDE solution, we show that the fixed cost coincides with the value function of the control problem and derive the corresponding optimal strategy.

\subsection{Wellposedness of the control problem}
In this section, we prove that the control problem is wellposed in the following sense (refer to \cite{Yong-Zhou-1999}): 
\begin{definition}\label{def:wellposed}
	We say that the control problem \eqref{cost-inf-continuous}-\eqref{state-inf-continuous} is wellposed if 
	\begin{equation}\label{finite-value}
		V_t( \mathcal X_{t-},i  ):=\inf_{\widetilde X\in\mathscr A} J(t,\mathcal X_{t-},i;\widetilde X)>-\infty.
	\end{equation}
\end{definition}
From the expression \eqref{cost-inf-continuous}, it is not immediately clear whether \eqref{finite-value} holds. To verify \eqref{finite-value}, we first reformulate the cost functional \eqref{cost-inf-continuous} by considering its finite-horizon counterpart and then letting the horizon tend to infinity.

\begin{lemma}\label{lemma:reformulation-T}
For any $\widetilde X\in\mathscr A$, define
\[
    J^T(t,\mathcal X_{t-},i;\widetilde X^T) = \mathbb E_t\left[\int_t^T \left(-\widetilde Y^T_{s-}\,d\widetilde X^T_s+ \frac{\gamma_s}{2}d[\widetilde X^T]_s - \widetilde\sigma^{\alpha_s}_s \,d[\widetilde X^T,W]_s\right)	+\int_t^T\left(\lambda^{\alpha_s}_s (\widetilde{X}^T_s)^2-\frac{\phi}{2}\widetilde Y^T_{s}\widetilde{X}^T_s\right)\,ds	 \right],
\]
where $\widetilde X^T_s=\widetilde X_s\mathbf{1}_{[t,T)}(s)$\footnote{ This is essentially a liquidation strategy. For a non-liquidation alternative, additional terms of the form $-(\widetilde Y_{T-}+\frac{\gamma_T}{2}\widetilde X_{T-}) \widetilde X_{T}-\frac{\gamma_T}{2}\widetilde X_{T}( \widetilde X_{T}-\widetilde X_{T-})$ in the finite-horizon cost needs to be introduced and can be shown to vanish due to the integrability \eqref{integrability:X} and \eqref{integrability:P}. The remainder of the analysis will follow a similar approach.} and $\widetilde Y^T$ is the corresponding state. 
	The cost with the finite-horizon and truncated strategy $\widetilde X^T$, $J^T(t,\mathcal X_{t-},i;\widetilde X^T)$, can be reformulated as 
	\begin{equation}\label{cost-T}
		\begin{split}
	J^{T}(t, \mathcal X_{t-},i ;\widetilde X^T)=&~\mathbb{E}_t\left[   \int_{t}^{T} \frac{1}{a_s^{\alpha_s}}(I_{s}^{\bar A,\alpha_s}\mathcal X_{s}+I^{\bar B,\alpha_s}_s)^{2}\,ds \right]+\mathcal X^\top_{t-} A^{i}_t\mathcal X_{t-} +\mathcal X^\top_{t-} B^{i}_t+C^{i}_t\\
	&+\mathbb E_t\left[ \left(\frac{1}{2}-\bar A^{\alpha_{T-}}_T\right) \gamma_T\widetilde{P}^2_{T} -\bar B^{\alpha_{T-}}_T\widetilde{P}_{T} -C^{\alpha_{T-}}_T 	\right],
	\end{split}
\end{equation}		
where the processes $I^{\bar A}$ and $I^{\bar B}$ are given by 
\begin{equation}\label{def:I-A}
	I^{\bar A,i}_t=\begin{pmatrix} -\gamma_t \mu_t\bar{A}_t^{i}-\gamma_t\sigma_{1,t}Z^{\bar{A},i}_t-\rho_t^i \gamma_t\bar{A}_t^{i}-\lambda_t^i \\
		-\rho_t^i \bar{A}_t^{i} +\rho_t^i-\mu_t\bar{A}_t^{i}    +\sigma_{1,t}^2\bar A^i_t   -Z^{\bar{A},i}_t \sigma_{1,t}+\frac{\mu_t}{2}+\frac{\phi}{2}-\frac{ \sigma_{1,t}^2}{2}\end{pmatrix}^\top,
\end{equation}
respectively,
\begin{equation}\label{def:I-B}
	I^{\bar B,i}_t =   -\frac{1}{2}\rho_t^{i} \bar{B}_{t}^{i}-\sigma_{1,t}\widetilde\sigma_t^i\bar{A}_t^{i}+\frac{1}{2}\sigma_{1,t}\widetilde\sigma_t^i       + \frac{1}{2}\sigma_{1,t}^2\bar B^i_t - \frac{1}{2}{\sigma}_{1,t} Z^{\bar B,i}_t - \frac{1}{2}\mu_t\bar B^i_t.
\end{equation} 
\end{lemma}
\begin{proof}
	 For any admissible strategy $\widetilde X$, noting that $(\widetilde X^T,\widetilde Y^T)=(\widetilde X,\widetilde Y)$ on $[0,T)$, we split the cost arising from the last jump from the cost:
     \begin{equation}
        \begin{split}
        &~\int_t^T \left(-\widetilde Y^T_{s-}\,d\widetilde X^T_s+ \frac{\gamma_s}{2}d[\widetilde X^T]_s - \widetilde\sigma^{\alpha_s}_s \,d[\widetilde X^T,W]_s\right)	+\int_t^T\left(\lambda^{\alpha_s}_s (\widetilde{X}^T_s)^2-\frac{\phi}{2}\widetilde Y^T_{s}\widetilde{X}^T_s\right)\,ds \\
        =&~\int_t^{T-} \left(-\widetilde Y_{s-}\,d\widetilde X_s+ \frac{\gamma_s}{2}d[\widetilde X]_s - \widetilde\sigma^{\alpha_s}_s \,d[\widetilde X,W]_s\right)	+\int_t^T\left(\lambda^{\alpha_s}_s \widetilde{X}_s^2-\frac{\phi}{2}\widetilde Y_{s}\widetilde{X}_s\right)\,ds\\
        &~-(\widetilde Y_{T-}-\frac{\gamma_T}{2}\Delta \widetilde X^T_T)\Delta \widetilde X^T_T\\
        =&~\int_t^{T-} \left(-\widetilde Y_{s-}\,d\widetilde X_s+ \frac{\gamma_s}{2}d[\widetilde X]_s - \widetilde\sigma^{\alpha_s}_s \,d[\widetilde X,W]_s\right)	+\int_t^T\left(\lambda^{\alpha_s}_s \widetilde{X}_s^2-\frac{\phi}{2}\widetilde Y_{s}\widetilde{X}_s\right)\,ds\\
        &~+(\widetilde Y_{T-}+\frac{\gamma_T}{2}  \widetilde X_{T-}) \widetilde X_{T-}  \quad (\textrm{since }\widetilde X^T_T =0)\\
        =&~ \int_t^{T-} \left(-\widetilde Y_{s-}\,d\widetilde X_s+ \frac{\gamma_s}{2}d[\widetilde X]_s - \widetilde\sigma^{\alpha_s}_s \,d[\widetilde X,W]_s\right)	+\int_t^T\left(\lambda^{\alpha_s}_s \widetilde{X}_s^2-\frac{\phi}{2}\widetilde Y_{s}\widetilde{X}_s\right)\,ds\\
        &~+(\widetilde Y_{T-}+\frac{\gamma_T}{2}  \widetilde X_{T-}) \widetilde X_{T-}  -
        (\mathcal X^\top_{T-} A^{\alpha_{T-}}_{T}\mathcal X_{T-} +\mathcal X^\top_{T-} B^{\alpha_{T-}}_T+C^{\alpha_{T-}}_T) \\
        &~+\mathcal X^\top_{T-} A^{\alpha_{T-}}_{T}\mathcal X_{T-} +\mathcal X^\top_{T-} B^{\alpha_{T-}}_T+C^{\alpha_{T-}}_T.
        \end{split}
     \end{equation}
Observe that 
\begin{equation}\label{eq:XAX-XB}
	\mathcal X^\top A^i\mathcal X = \gamma \bar A^i \left( \widetilde X + \frac{1}{\gamma}\widetilde Y		   \right)^2 - \frac{1}{2\gamma}\widetilde Y^2=\gamma \bar A^i   \widetilde P^2- \frac{1}{2\gamma}\widetilde Y^2,\quad \mathcal X^\top B^i = \bar B^i \widetilde X+\frac{\bar B^i}{\gamma}\widetilde Y =  \bar B^i \widetilde P,
\end{equation}
	where we recall that $\widetilde P:=\widetilde X+\frac{1}{\gamma}\widetilde Y$, that $\bar A$ is the solution to \eqref{BSDE:A-inf} constructed in Proposition \ref{prop:A}, and that $\bar B$ is the solution to \eqref{BSDE:B-inf} constructed in Theorem \ref{thm:B}. Note that both $\gamma \bar A^\alpha   \widetilde P^2$ and $\bar B^\alpha \widetilde P$ satisfy the condition in \cite[Lemma 4.3]{Hu-Liang-Tang}. Therefore, applying integration by parts formula in \cite[Lemma 4.3]{Hu-Liang-Tang} and \cite[Theorem 36]{Protter-2005} to $\mathcal X^\top A^{\alpha} \mathcal X  +\mathcal X^\top  B^{\alpha } +C^{\alpha } $, as well as similar analysis as in \cite[Proposition 4.1]{FHX-2023}, we obtain 
\begin{equation}\label{eq:cost-rewritten-1}
    \begin{split}
        &~\int_t^{T-} \left(-\widetilde Y_{s-}\,d\widetilde X_s+ \frac{\gamma_s}{2}d[\widetilde X]_s - \widetilde\sigma^{\alpha_s}_s \,d[\widetilde X,W]_s\right)	+\int_t^T\left(\lambda^{\alpha_s}_s \widetilde{X}_s^2-\frac{\phi}{2}\widetilde Y_{s}\widetilde{X}_s\right)\,ds\\
        &~+(\widetilde Y_{T-}+\frac{\gamma_T}{2}  \widetilde X_{T-}) \widetilde X_{T-}  -
        (\mathcal X^\top_{T-} A^{\alpha_{T-}}_{T}\mathcal X_{T-} +\mathcal X^\top_{T-} B^{\alpha_{T-}}_T+C^{\alpha_{T-}}_T) \\
        &~+\mathcal X^\top_{T-} A^{\alpha_{T-}}_{T}\mathcal X_{T-} +\mathcal X^\top_{T-} B^{\alpha_{T-}}_T+C^{\alpha_{T-}}_T\\
        =&~   \int_t^{T-} \left(-\widetilde Y_{s-}\,d\widetilde X_s+ \frac{\gamma_s}{2}d[\widetilde X]_s - \widetilde\sigma^{\alpha_s}_s \,d[\widetilde X,W]_s\right)	+\int_t^T\left(\lambda^{\alpha_s}_s \widetilde{X}_s^2-\frac{\phi}{2}\widetilde Y_{s}\widetilde{X}_s\right)\,ds          \\
        &~ +  \left(\frac{1}{2}-\bar A_T^{\alpha_{T-}}\right) \gamma_T\widetilde{P}^2_{T} -\bar B_T^{\alpha_{T-}}\widetilde{P}_{T} -C^{\alpha_{T-}}_T               \\
        &~   +\mathcal X^\top_{T-} A^{\alpha_{T-}}_{T}\mathcal X_{T-} +\mathcal X^\top_{T-} B^{\alpha_{T-}}_T+C^{\alpha_{T-}}_T           \\
        =&~ \int_t^T \frac{1}{a^{\alpha_s}_s} \left(I^{\bar A,\alpha_s}_s\mathcal X_s+I^{\bar B,\alpha_s}_s\right)^2\,ds +\int_t^T \mathcal X_{s}^{\top} Z^{A,\alpha_s}_s \mathcal X_{s} \,dW_s +\int_t^T \mathcal X_{s}^{\top} Z^{B,\alpha_s}_s  \,dW_s+\int_t^T Z^{C,\alpha_s}_s  \,dW_s\\
			&~+\int^{T}_{t} (A_s^{\alpha_s}\mathcal X_{s})^{\top}\mathcal D^{\alpha_s}_s  \,dW_s+\int^{T}_{t} (B_s^{\alpha_s})^{\top}\mathcal D^{\alpha_s}_s  \,dW_s+\int_t^T \sum_{j,j'\in\cal M} \gamma^{}_s\widetilde P_{s}^2 ( \bar A^j_{s} - \bar A^{j'}_{s} )\mathbf{1}_{\{  \alpha_{s-}=j' \}} \,d\widetilde N^{j'j}_s\\
			&~+ \mathcal X^\top_{t-}A^i_t\mathcal X_{t-} + \mathcal X^\top_{t-}B^i_t+C^i_t+\left(\frac{1}{2}-\bar A_T^{\alpha_{T-}}\right) \gamma_T\widetilde{P}^2_{T} -\bar B_T^{\alpha_{T-}}\widetilde{P}_{T} -C^{\alpha_{T-}}_T.
    \end{split}
\end{equation}
Taking conditional expectations, we obtain \eqref{cost-T}.  
\end{proof}
Next, we reformulate the cost functional \eqref{cost-inf-continuous} by letting $T\rightarrow\infty$ in \eqref{cost-T}.
\begin{theorem}\label{thm:reformulation-cost-inf}
	For any $\widetilde X\in\mathscr A$, the cost functional for the control problem \eqref{cost-inf-continuous}-\eqref{state-inf-continuous} starting from $(t,\mathcal X_{t-},i)$ can be rewritten as
	\begin{equation}\label{eq:cost-inf}
		\begin{split}
			J(t,\mathcal X_{t-},i; \widetilde X)
			=&~		 \mathbb E_t\left[ 	\int_t^\infty \frac{1}{a^{\alpha_s}_s} \left(I^{A,\alpha_s}_s\mathcal X_s+I^{B,\alpha_s}_s\right)^2\,ds 	 	\right] + \mathcal X^\top_{t-}A_t^{i}\mathcal X_{t-} + \mathcal X^\top_{t-}B_t^{i}+C_t^{i},
		\end{split}
	\end{equation}
	where we recall ${\cal X}=(\widetilde X,\widetilde Y)^\top$.
	
	In particular, $\inf_{\widetilde X\in\mathscr A} J(t,\mathcal X_{t-},i;\widetilde X)>-\infty$, and infinite-horizon control problem \eqref{cost-inf-continuous}-\eqref{state-inf-continuous} is wellposed in the sesne of Definition \ref{def:wellposed}. Moreover, the optimal strategy is unique whenever it exists.
\end{theorem}
\begin{proof}
	First, by monotone convergence theorem  we have the convergence of the first term in \eqref{cost-T}
	\begin{equation}\label{estimate:diff-cost-k-Part-1}
		\begin{aligned}
			\lim_{T\rightarrow\infty}\mathbb E_t\left[  \int_t^T   \frac{1}{a^{\alpha_s}_s} \left(  I_{s}^{A,\alpha_s}\mathcal X_{s}+  I^{B,\alpha_s}_s\right)^2\,ds \right]=\mathbb E_t\left[  \int_t^\infty   \frac{1}{a^{\alpha_s}_s} \left(  I_{s}^{A,\alpha_s}\mathcal X_{s}+  I^{B,\alpha_s}_s\right)^2\,ds \right].
		\end{aligned}
	\end{equation}
Second, due to \eqref{integrability:P},  Proposition~\ref{prop:A}, Lemma~\ref{lemma:gamma}, Theorem~\ref{thm:B} and Theorem \ref{thm:C}, we obtain the following convergence
\begin{equation}
	\begin{split}
		&~\left| \mathbb E_t\left[\left(\frac{1}{2}-\bar A_T^{\alpha_{T-}}\right) \gamma_T\widetilde{P}^2_{T} -\bar B_T^{\alpha_{T-}}\widetilde{P}_{T} -C^{\alpha_{T-}}_{T}\right] \right| \\
		\leq&~ ce^{-2\varphi T}\mathbb E_t\left[\sup_{s\ge t}e^{2 \varphi s}\gamma_s\widetilde{P}^2_{s} \right] + ce^{- \frac{\beta}{4} T}\|\bar B\|_{\infty,-\frac{\beta}{4}}\mathbb E_t\left[\sup_{s\ge t}e^{2 \varphi s} \gamma_s \widetilde{P}^2_{s}\right]^{1/2} \mathbb E_t\left[\sup_{s\ge t} \frac{1}{\gamma_s} \right]^{1/2}\\
		&~+ce^{-\frac{\beta T}{8} }\mathbb E_t\left[\sup_{s\ge t}e^{\frac{\beta s}{4} }\left|C_s\right|^2\right]^{1/2}\\
		\rightarrow&~0,\qquad \text{as }T\rightarrow\infty.
	\end{split}
\end{equation}
Therefore, the right hand side of \eqref{cost-T} converges to the right hand side of \eqref{eq:cost-inf}. It remains to verify the convergence of the left hand side.

To do so, note the following convergence
\begin{equation}\label{convergence:XY}
	\begin{split}
		&~\left| \mathbb E_t\left[\frac{1}{2} \gamma_T\widetilde{X}^2_{T-} +\widetilde{X}_{T-} \widetilde{Y}_{T-}\right] \right|  
        = \left| \mathbb E_t\left[ -\frac{1}{2}\gamma_T\widetilde{X}^2_{T-} +\gamma_T\widetilde{X}_{T-} \widetilde{P}_{T}\right] \right|\\
		\leq&~ ce^{-2\varphi T}\mathbb E_t\left[\sup_{s\ge t}e^{2 \varphi s}\gamma_s\widetilde{X}^2_{s} \right] + ce^{-2\varphi T}\mathbb E_t\left[\sup_{s\ge t}e^{2 \varphi s}\gamma_s\widetilde{X}^2_{s} \right]^{1/2} \mathbb E_t\left[\sup_{s\ge t}e^{2 \varphi s}\gamma_s\widetilde{P}^2_{s} \right]^{1/2}\\
		\rightarrow&~0,\qquad \text{as }T\rightarrow\infty.
	\end{split}
\end{equation}
From the integrability \eqref{integrability:X} and \eqref{integrability:P}, the supremum over $T\geq 0$ of all the stochastic integrals in \eqref{eq:cost-rewritten-1} is integrable.  By the dominated convergence theorem and \eqref{convergence:XY}, we get the convergence of the left hand side and thus the equality \eqref{eq:cost-inf}.
\end{proof}

\subsection{Existence of an optimal strategy}

We construct the candidate optimal strategy in the following order: 

$\bullet$ Define the  $\ell$-dimentional processes $\pi=(\pi_s)_{s\geq 0}$ in terms of $(\bar{A},Z^{\bar A})$ by
\[
	\pi^{i}= \frac{\gamma}{a^i}    \left(\frac{\lambda^{i}}{\gamma}+(\rho^{i} +\mu)\bar A^{i}+\sigma_{1} Z^{\bar A,i}\right) ,\quad  i\in \cal M.
\]
$\bullet$ In terms of $\pi$, define the process $\widetilde P^*=(\widetilde P^*)_{s\geq 0}$ as
\begin{equation}\label{SDE:P*}
	\left\{\begin{aligned}
		d\widetilde P^{*}_s=&~-\left(\frac{\phi}{2}+\left(\rho^{\alpha_s}_s+\mu_s-(\sigma_{1,s})^2\right)\pi^{\alpha_s}_s\right)\widetilde P^{*}_s\,ds\\
		&~-\left((\rho^{\alpha_s}_s+\mu_s-(\sigma_{1,s})^2)\eta^{\alpha_s}_s+\frac{\widetilde\sigma^{\alpha_s}_s\sigma_{1,s}}{\gamma_s}\right)\,ds \\
		&~-\left(\sigma_{1,s}\pi^{\alpha_s}_s\widetilde P^{*}_s+\sigma_{1,s}\eta^{\alpha_s}_s-\frac{\widetilde\sigma^{\alpha_s}_s}{\gamma_s} \right) \,dW_s,\\
		\widetilde P^{*}_0=&~x_0,
	\end{aligned}\right. 
\end{equation}
where
\[ 
	\eta^{i}_s:=\frac{1}{2a^{i}_s}  \left\{   \left(\rho_s^{i} +\mu_s-\sigma_{1,s}^2  \right) \bar B^{i}_{s}+2\sigma_{1,s}\widetilde\sigma_s^{i}\bar A_s^{i}-\sigma_{1,s}\widetilde\sigma_s^{i}   	 +\sigma_{1,s} Z^{\bar B,i}_s  	  \right\},\quad s\geq 0.
\] 
$\bullet$ In terms of $\widetilde P^*$ and $\pi$, define processes $\widetilde X^*$ and $\widetilde Y^*$ on $[0,\infty)$ as 
	\begin{equation}\label{optimal-state-inf}
	\begin{aligned}
		\widetilde X^{*}=\left(1-\pi^{\alpha}\right)  \widetilde P^{*} - \eta^{\alpha}   \quad \text{ and }\quad   \widetilde Y^{*}=\gamma\pi^{\alpha} \widetilde P^{*} + \gamma \eta^{\alpha},
	\end{aligned}
\end{equation}
$\bullet$ Recall the initial state $\mathcal X_{0-}=(x_0, y_0)^\top$ and let the initial regime be $i_0$. Define $\Delta\widetilde X^*_0$ as  
\begin{equation}\label{jump-inf}
	\Delta  \widetilde{X}^*_{0}= \frac{I^{\bar A,i_0}_0}{a^{i_0}_0} {\mathcal X}_{0-}+\frac{I^{\bar B,i_0}_0}{a^{i_0}_0},
\end{equation}
where $I^{\bar A}$ and $I^{\bar B}$ are given by \eqref{def:I-A} and \eqref{def:I-B}, respectively. 

The next theorem proves that $(\widetilde X^*)_{s\geq 0}$ together with $\Delta X_0^*$ is the unique optimal strategy of \eqref{cost-inf-continuous}-\eqref{state-inf-continuous} under the following additional assumption:

\begin{ass}\label{ass:existence-control}
	For some $0<K_\pi<L$, $ |\pi|^2\in H^{2,K_\pi}_{\text{BMO}}\bigcap  S^{2,0}$ and $\eta\in S^{4,-\varphi}$.

\end{ass}
\begin{remark}
Assumption~\ref{ass:existence-control} seems strong, but it plays a crucial role in verifying that \( \widetilde{X}^* \in \mathscr{A} \). Essentially, it requires that \( |\sigma_1 Z^{\bar{A}}|^2 \in H^{2,K_\pi}_{\text{BMO}}  \bigcap S^{2,0}   \) and that \( \sup_t |\sigma_{1,t} Z^{\bar{B}}_t| \) is integrable to the fourth power. A similar assumption is also made in \cite[Theorem 3.4 and Remark 3.5(b)]{AKU-2021}. These conditions are trivially satisfied when \( \sigma_1 = 0 \) or $(Z^{\bar A},Z^{\bar B})$ are bounded.
\end{remark}

\begin{theorem}\label{thm:existence-optimal-control}
	Let Assumption \ref{ass:standing} and Assumption \ref{ass:existence-control} hold. 
	\begin{itemize}
		\item[i)] The value function defined in Definition \ref{def:wellposed}, starting at $0$ for the infinite-horizon control problem \eqref{cost-inf-continuous}-\eqref{state-inf-continuous}, satisfies
		\begin{equation} \label{value-inf}
			V_0(\mathcal X_{0-},i_0)=\mathcal X^\top_{0-} A^{i_0}_0 \mathcal X_{0-}+\mathcal X^\top_{0-} B^{i_0}_0+C^{i_0}_0. 
		\end{equation}

	\item[ii)]	If $\widetilde X^*$ defined by \eqref{optimal-state-inf}-\eqref{jump-inf} is a c\`adl\`ag semimartingale, then the pair $( X^*_t, Y^*_t):=(e^{\frac{\phi}{2}t}\widetilde X^*_t, e^{\frac{\phi}{2}t} \widetilde Y^*_t)$ is the unique optimal control-state pair for the infinite-horizon control problem \eqref{cost-inf-continuous}-\eqref{state-inf-continuous}.

	\end{itemize}
\end{theorem}

\begin{remark}
	If $\widetilde X^*$ is not a semimartingale, then the value function \eqref{value-inf} can only be approximated but cannot be reached by any admissible strategy in $\mathscr A$.
\end{remark}

The following lemma plays a crucial role in establishing the admissible integrability of \( \widetilde{X}^* \) and \( \widetilde{P}^* \).
\begin{lemma}\label{lemma:Gamma}
	Let $\zeta$ and $\theta$ be two processes such that $\zeta\in H^{2,K_{\zeta}}_{\text{BMO}}$ and $\theta^2\in H^{2,K_{\theta}}_{\text{BMO}}$ with $K_{\theta}<0$. For any $K>K_{\zeta}$ and $t\geq 0$ denote  
	$$\Gamma_t:=\exp\left\{\int^t_0e^{-Ks}\zeta_s \,ds+\int^t_0\theta_s\,dW_s\right\}.$$
	We have for any $p\ge 1$
	$$
	\mathbb E\left[\sup_{t\ge 0}|\Gamma_t|^p\right]<\infty.
	$$
\end{lemma}
\begin{proof}
	By H\"older's inequality, we have
	\begin{equation*}
		\begin{aligned}
			\mathbb E\left[\sup_{t\ge 0}|\Gamma_t|^p\right]
			\leq \mathbb E\left[\sup_{t\ge 0}\exp\left\{2p\int^t_0e^{-Ks}\zeta_s\, \,ds+p^2\int^t_0\theta_s^2 \,ds\right\}\right]^{1/2}\mathbb E\left[\sup_{t\geq 0} L_t^2\right]^{1/2},
		\end{aligned}
	\end{equation*}
	where
	\[
	L_t:=\exp\left\{-\int^t_0\frac{(p\theta_s)^2}{2}\,ds+ \int^t_0 p\theta_s\,dW_s\right\}.
	\]
	Next, we prove the above two expectations are finite.
	
	First,	since $\zeta\in H^{2,K_{\zeta}}_{\text{BMO}}$ and $\theta^2\in H^{2,K_{\theta}}_{\text{BMO}}$, by H\"older's inequality and applying Lemma \ref{lemma:zeta-BMO} to $\zeta$ and $\theta^2$, we obtain 
	\[\begin{aligned}
		&~\mathbb E\left[\sup_{t\ge 0}\exp\left\{2p\int^t_0e^{-Ks}\zeta_s \,ds+p^2\int^t_0\theta_s^2 \,ds\right\}\right]\\
		\leq&~\mathbb E\left[\exp\left\{4p\int^\infty_0e^{-Ks}|\zeta_s| \,ds\right\}\right]^{1/2}\mathbb E\left[  \exp\left\{    2p^2\int^\infty_0\theta_s^2 \,ds \right\}   \right]^{1/2}\\
		<&~\infty.
	\end{aligned}\]
	Second,	by Novikov's criterion and Lemma \ref{lemma:zeta-BMO}, $L_s$ is a true martingale. It follows from Doob's maximal inequality and H\"older's inequality that
	\begin{align*}
			&~\mathbb E\left[\sup_{t\ge 0}\left(\exp\left\{-\int^t_0\frac{(p\theta_s)^2}{2}\,ds+\int^t_0p\theta_s\,dW_s\right\}\right)^2\right]\\
			\leq&~ 4\mathbb E\left[   \exp\left\{   -\int_0^\infty (p\theta_s)^2\,ds  +2 \int_0^\infty p\theta_s\,dW_s  \right\}      \right]\\
			\leq&~\mathbb E\left[\exp\left\{  4\int_0^\infty p\theta_s\,dW_s-8\int_0^\infty(p\theta_s)^2\,ds     \right\} 				   \right]^{1/2}  \mathbb E\left[  \exp\left\{  6p^2\int_0^\infty \theta^2_s\,ds     \right\}  \right]^{1/2}\\
			=&~	\mathbb E\left[  \exp\left\{  6p^2\int_0^\infty \theta^2_s\,ds     \right\}  \right]^{1/2} 	\\
			<&~\infty,
	\end{align*}
	where the last step is given by Lemma \ref{lemma:zeta-BMO}.
\end{proof}
Using Lemma~\ref{lemma:Gamma}, we can establish the integrability result stated in the next lemma.
\begin{lemma}
	Let Assumption \ref{ass:standing} and Assumption \ref{ass:existence-control} hold.
The processes $\widetilde P^*$ and $\widetilde X^*$ defined in \eqref{SDE:P*} and \eqref{optimal-state-inf} satisfy the integrability \eqref{integrability:X} and \eqref{integrability:P}.
\end{lemma}
\begin{proof}
	We begin by deriving an estimate of $\gamma_s\widetilde P^*$, which satisfies the SDE: 
	\begin{equation*}
	\left\{\begin{aligned}
		d(\gamma_s\widetilde P^{*}_s)=&~ \left(\bar q ^{\alpha_s}_s\gamma_s\widetilde P^{*}_s -(\rho^{\alpha_s}_s+\mu_s)\gamma_s\eta^{\alpha_s}_s\right)\,ds +\Big(\sigma_{1,s}(1-\pi^{\alpha_s}_s)\gamma_s\widetilde P^{*}_s-\sigma_{1,s}\gamma_s\eta^{\alpha_s}_s+\widetilde\sigma_s^{\alpha_s}\Big)\,dW_s,\\
		\gamma_0\widetilde P^*_0=&~\gamma_0x_0,
	\end{aligned}\right. 
\end{equation*}
where
\begin{equation*}
	\begin{split}
		\bar q ^{\alpha_s}_s :=&~ \mu_s-\left(\frac{\phi}{2}+(\rho^{\alpha_s}_s+\mu_s)\pi_s^{\alpha_s}\right)\\
		=&~	\mu_s-(\rho_s^{\alpha_s}+\mu_s)\frac{ \gamma_s\sigma_{1,s} Z^{\bar A,\alpha_s}_s }{a_s^{\alpha_s}}	-\frac{\phi}{2}  - \frac{  \lambda_s^{\alpha_s}(\rho_s^{\alpha_s}+\mu_s)+\gamma_s(\rho_s^{\alpha_s}+\mu_s)^2\bar A_s^{\alpha_s} }{a^{\alpha_s}_s} 	\\
		:=&~ \mu_s-(\rho_s^{\alpha_s}+\mu_s)\frac{ \gamma_s\sigma_{1,s} Z^{\bar A,\alpha_s}_s }{a_s^{\alpha_s}} - \hat q^{\alpha_s}_s. 
	\end{split}
\end{equation*}
Note that the signs of the first two terms in \(\bar q^\alpha \) are indeterminate, and the coefficient of the linear term in the diffusion, that is, $\sigma_1 (1-\pi^\alpha)\gamma$, is unbounded. To estimate \( \gamma \widetilde{P}^* \), we first introduce two transformations that yield an SDE with a monotone driver and without unbounded linear term in the diffusion. 

To eliminate the first two terms of \(\bar q^\alpha \) from the driver of \( \gamma \widetilde{P}^* \), we consider the transformation \( \Gamma \gamma \widetilde{P}^* \), where
\[
   \Gamma_\cdot=\exp\left\{\int^\cdot_0\left(-\mu_s + \frac{(\rho^{\alpha_s}_s+\mu_s) \gamma_s \sigma_{1,s}Z_s^{\bar{A},
		\alpha_s}}{  a^{\alpha_s}_s  }\right)\,ds\right\}.
\]
It follows that 
\begin{equation*}
	d(\Gamma_s\gamma_s\widetilde P^*_s) =  \left(-\hat q ^{\alpha_s}_s  \Gamma_s \gamma_s\widetilde P^{*}_s -\Gamma_s(\rho^{\alpha_s}_s+\mu_s)\gamma_s\eta^{\alpha_s}_s\right)\,ds + \Big(\sigma_{1,s}(1-\pi^{\alpha_s}_s) \Gamma_s  \gamma_s\widetilde P^{*}_s-\Gamma_s\sigma_{1,s}\gamma_s\eta^{\alpha_s}_s+\Gamma_s\widetilde\sigma_s^{\alpha_s} \Big)\,dW_s.
\end{equation*}
%
Next, to eliminate the first term in the diffusion, we consider the transformation \( \Lambda := \frac{\Gamma}{\hat{\mathcal{E}}} \gamma \widetilde{P}^* \), where
\begin{align*}
	\hat{\mathcal E}:=\mathcal E\left(\sigma_{1}(1-\pi^{\alpha})\right).
\end{align*}
It follows that 
\begin{equation*}
	\begin{split}
		d\Lambda_s=&~- \hat q^{\alpha_s}_s \Lambda_s\,ds -\frac{\Gamma_s}{\hat{\mathcal E}_s}(\rho^{\alpha_s}_s+\mu_s)\gamma_s\eta^{\alpha_s}_s\,ds-\frac{\Gamma_s}{\hat{\mathcal E}_s}\sigma_{1,s}(1-\pi^{\alpha_s}_s)(-\sigma_{1,s}\gamma_s\eta^{\alpha_s}_s+\widetilde\sigma_s^{\alpha_s} )\,ds\\
		&~-\frac{\Gamma_s}{\hat{\mathcal E}_s} (\sigma_{1,s}\gamma_s\eta^{\alpha_s}_s-\widetilde\sigma_s^{\alpha_s})\,dW_s.
	\end{split}
\end{equation*}
For $\frac{\Gamma}{\hat{\cal E}}$, we have the following expression:
\begin{equation*}
	\frac{\Gamma}{\hat{\mathcal E}} = \exp\left( \int_0^t  \left\{ -\mu_s + \frac{(\rho^{\alpha_s}_s+\mu_s)\gamma_s\sigma_{1,s} Z^{\bar A,\alpha_s}_s			}{a^{\alpha_s}_s} 	+\frac{1}{2}\sigma_{1,s}^2(1-\pi_s)^2 \right\}\,ds - \int_0^t\sigma_{1,s}(1-\pi_s)\,dW_s   \right).
\end{equation*}
Recall $0<K_\pi <L$. Then for any constant $K_\theta$ such that $ K_\pi - 2L <K_\theta<0$, it holds that by Assumption \ref{ass:existence-control}
\begin{equation*}
	\begin{split}
		\esssup_{\tau,\omega}\mathbb E_\tau \left[ \int_\tau^\infty e^{-2K_\theta s} \{   \sigma_{1,s}^2(1-\pi_s)^2   \}^2\,ds  \right] \leq&~ c\esssup_{\tau,\omega}\mathbb E_\tau \left[ \int_\tau^\infty e^{-2K_\theta s-4Ls+2K_\pi s}   e^{-2K_\pi s } \{ (1-\pi_s)^2   \}^2\,ds  \right] \\
		 \leq&~ c\esssup_{\tau,\omega}\mathbb E_\tau \left[ \int_\tau^\infty   e^{-2K_\pi s } \{ (1-\pi_s)^2   \}^2\,ds  \right] <\infty.
	\end{split}
\end{equation*}
Moreover, by choosing $0<K_{\bar A}<L$ we have
\[
	\left|-\mu_s + \frac{(\rho^{\alpha_s}_s+\mu_s)\gamma_s\sigma_{1,s} Z^{\bar A,\alpha_s}_s			}{a^{\alpha_s}_s} 	+\frac{1}{2}\sigma_{1,s}^2(1-\pi_s)^2  \right|  \leq ce^{-Ls} \left\{ 1 + |Z^{\bar A,\alpha_s}_s|  + |\pi_s|^2			  \right\},
\]
where $1 +  Z^{\bar A,\alpha}   + |\pi|^2	\in H^{2,K_{\bar A}\vee K_\pi}_{\text{BMO}}$ and $K_\pi\vee K_{\bar A}<L$. Thus, $\frac{\Gamma}{\hat{\cal E}}$ satisfies all conditions in Lemma \ref{lemma:Gamma}, which implies that for each $p\geq 1$
\begin{equation}\label{estimate:Gamma/E}
	\begin{split}
		\mathbb E\left[ \sup_{t\geq 0}  \left(\frac{\Gamma_t}{\hat{\cal E}_t}  \right)^p \right]<\infty.
	\end{split}
\end{equation}
Similary, under the same requirement of $K_\pi$ and $K_{\bar A}$, we also have 
\begin{equation}\label{estimate:E/Gamma}
	\begin{split}
		\mathbb E\left[ \sup_{t\geq 0}  \left(\frac{ \hat{\cal E}_t}{\Gamma_t}   \right)^p \right]<\infty.
	\end{split}
\end{equation}
Note that $\hat q^i$ satisfies the following estimate:
\begin{align*}
	&~-\hat q^i\\
\leq&~ 	 	-\frac{\phi}{2}  - \frac{  \lambda_s^{\alpha_s}(\rho_s^{\alpha_s}+\mu_s)  }{ a^{\alpha_s}_s  } 					\qquad(\text{since }\bar A\text{ is positively valued})	\\ 		
\leq&~ 	 	-\frac{\phi}{2}  - \frac{  \lambda_s^{\alpha_s}(\rho_s^{\alpha_s}+\mu_s)  }{  \frac{\mu_s\gamma_s}{2} +\gamma_s\rho_s^{\alpha_s}+\lambda_s^{\alpha_s}+\frac{\phi\gamma_s}{2}    } 			\qquad(\text{since }\bar A\text{ is valued below }1/2)	\\ 		
=&~ 	-\frac{\phi}{2}  
- \frac{  \frac{\lambda_s^{\alpha_s}}{\gamma_s}(\rho_s^{\alpha_s}+\mu_s)  }{  \frac{\mu_s }{2} + \rho_s^{\alpha_s}+\frac{\lambda_s^{\alpha_s}}{\gamma_s}+\frac{\phi}{2}   }.
\end{align*}
 Recalling that by Assumption \ref{ass:standing} (iii) and the definition of $\varphi$, we have
	\[
		\frac{\phi}{2} + \min_i \frac{
		\frac{\lambda_s^{i}}{\gamma_s}(\rho_s^{i}+\mu_s) }{  \frac{\mu_s}{2} + \rho_s^{i} +
		\frac{\lambda_s^{i}}{\gamma_s} + \frac{\phi}{2}}-\varphi>0 \quad \text{ and } \quad  \frac{\beta}{8}-\varphi >0,
	\] 
	In particular, $\hat q^{i}-\varphi>0$ for all $i\in\cal M$.

	By It\^o's formula, it holds that 
	\begin{align*}
		&~e^{2\varphi t} \Lambda^2_t + \int^t_02e^{2\varphi s}\Lambda_s^2\left( \hat q^{\alpha_s}_s  -   \varphi  \right)\,ds \\
		=&~ \Lambda_0^2 - \int_0^t 2 e^{2\varphi s}\Lambda_s \frac{\Gamma_s}{\hat{\mathcal E}_s}\left\{   	(\rho^{\alpha_s}_s+\mu_s  )\gamma_s\eta_s^{\alpha_s}  + \sigma_{1,s}(1-\pi_s^{\alpha_s}) ( -\sigma_{1,s}\gamma_s\eta_s^{\alpha_s} + \widetilde\sigma^{\alpha_s}_s 			)  	\right\}\,ds \\
		&~+\int_0^t e^{2\varphi s}\left( \frac{\Gamma_s}{\hat{\mathcal E}_s}  \right)^2\left(  \sigma_{1,s}\gamma_s\eta_s^{\alpha_s} - \widetilde\sigma_s^{\alpha_s}		  \right)^2\,ds \\
		&~-\int_0^t   2e^{2\varphi s} \Lambda_s \frac{\Gamma_s}{\hat{\mathcal E}_s} (	\sigma_{1,s}\gamma_s\eta_s^{\alpha_s}  - \widetilde\sigma^{\alpha_s}_s		)\,dW_s.
	\end{align*}
Taking $\mathbb E \left[\sup_{t\ge 0}(\dots)^p\right]$ for any $p\ge 1$ on both sides, using $\hat q^i>\varphi$ and BDG's inequality in the first inequality, and using Young's inequality in the second one, we have
	\begin{align*}
		&~\mathbb E\left[\sup_{t\ge 0}e^{2p\varphi t}|\Lambda_t|^{2p}\right]\\
		\leq&~c\mathbb E\left[\Lambda_0^{2p}\right]+c\mathbb E\left[\left(\int^\infty_0 e^{2\varphi s}\left|\Lambda_s\Gamma_s\hat{\mathcal E}_s^{-1}(\rho^{\alpha_s}_s+\mu_s)\gamma_s\eta^{\alpha_s}_s\right|ds\right)^p\right]\\
		&~+c\mathbb E\left[\left(\int^\infty_0e^{2\varphi s}\left|\Lambda_s\Gamma_s\hat{\mathcal E}^{-1}_s\sigma_{1,s}(1-\pi^{\alpha_s}_s)(-\sigma_{1,s}\gamma_s\eta^{\alpha_s}_s+\widetilde\sigma_s^{\alpha_s} )\right|\,ds\right)^p\right]\\
		&~+c\mathbb E\left[  \left(	\int_0^\infty e^{2\varphi s}\left( \Gamma_s \hat{\mathcal E}_s^{-1}  \right)^2\left(  \sigma_{1,s}\gamma_s\eta_s^{\alpha_s} - \widetilde\sigma_s^{\alpha_s}	  \right)^2\,ds	  \right)^p     \right]\\
		&~+c\mathbb E\left[\left|\int^\infty_0\left\{  e^{2\varphi s}\Lambda_s\Gamma_s\hat{\mathcal E}_s^{-1}(\sigma_{1,s}\gamma_s\eta^{\alpha_s}_s-\widetilde\sigma_s^{\alpha_s} )\right\}^2\,ds\right|^{p/2}\right] 			\\
		\leq&~c \mathbb E\left[ \Lambda^{2p}_0\right]+\frac{1}{4}\mathbb E\left[\sup_{t\ge 0}e^{2p\varphi t}|\Lambda_t|^{2p}\right]+c\mathbb E\left[\left(\int^\infty_0e^{ \varphi s}\left|\Gamma_s\hat{\mathcal E}_s^{-1}(\rho^{\alpha_s}_s+\mu_s)\gamma_s\eta^{\alpha_s}_s\right|ds\right)^{2p}\right]\\
		&~+c\mathbb E\left[\left(\int^\infty_0e^{\varphi s}\left|\Gamma_s\hat{\mathcal E}_s^{-1}\sigma_{1,s}(1-\pi^{\alpha_s}_s)(-\sigma_{1,s}\gamma_s\eta^{\alpha_s}_s+\widetilde\sigma_s^{\alpha_s} )\right|ds\right)^{2p}\right]\\
		&~+c\mathbb E\left[  \left(	\int_0^\infty e^{2\varphi s}\left( \Gamma_s \hat{\mathcal E}_s^{-1}  \right)^2\left(  \sigma_{1,s}\gamma_s\eta_s^{\alpha_s} - \widetilde\sigma_s^{\alpha_s} 		  \right)^2\,ds	  \right)^p     \right]\\
		&~+c\mathbb E\left[\left(   \int^\infty_0\left\{  e^{\varphi s} \Gamma_s\hat{\mathcal E}_s^{-1}(\sigma_{1,s}\gamma_s\eta^{\alpha_s}_s-\widetilde\sigma_s^{\alpha_s} )\right\}^2\,ds\right)^{p}\right],
	\end{align*}
which further implies that by the assumption $\varphi<\frac{\beta}{8}$, \eqref{estimate:Gamma/E} and Corollary \eqref{coro:gamma-a}
\begin{align*}
		\mathbb E\left[\sup_{t\ge 0}e^{2p\varphi t}|\Lambda_t|^{2p}\right]   
	\leq&~c \mathbb E\left[\Lambda^{2p}_0 \right]+c\left( \|\bar B\|^{2p}_{\infty,-\frac{\beta}{4}} +1\right)   \mathbb E\left[\sup_{t\ge 0}\left|\Gamma_t\hat{\mathcal E}^{-1}_t\right|^{2p}  \left(\int_0^\infty e^{-(\frac{\beta}{4}-\varphi)s}   \,ds \right)^{2p}    \right]\\
	&~+	 	c\mathbb E\left[  \sup_{t\ge 0}\left|\Gamma_t\hat{\mathcal E}^{-1}_t\right|^{4p}  \right]^{1/2} \mathbb E\left[  \left(\int_0^\infty e^{-2(\frac{\beta}{8}-\varphi)s}    \,ds\right)^{2p}    \left(  \int_0^\infty e^{\frac{\beta}{4}s}	 | Z^{ \bar B}_s |^2  \,ds \right)^{2p}	\right]^{1/2}					\\
	&~+c\mathbb E\left[\sup_{t\ge 0}\left|\Gamma_t\hat{\mathcal E}^{-1}_t\right|^{2p}\left(\int^\infty_0e^{-(\frac{\beta}{4}-\varphi) s} ( e^{-Ls}|Z^{\bar A}_s|+1  )  \,ds\right)^{2p} \right]  \left( \|\bar B\|^{2p}_{\infty,-\frac{\beta}{4}} +1\right)  \\
	&~+ c\mathbb E\left[  \sup_{t\ge 0}\left|\Gamma_t\hat{\mathcal E}^{-1}_t\right|^{4p}   \right]^{1/2}	\mathbb E\left[ \left(\int_0^\infty e^{-Ls}|Z_s^{\bar A}|^2\,ds    \right)^{4p}  \right]^{1/4}	 \mathbb E\left[  \left(  \int_0^\infty	  e^{\frac{\beta}{4}s}| Z^{\bar B}_s |^2\,ds   \right)^{4p}  \right]^{1/4}	 \\
	&~+c\mathbb E\left[\sup_{t\ge 0}\left|\Gamma_t \hat{\mathcal E}^{-1}_t\right|^{2p}  \left( \int_0^\infty e^{-(\frac{\beta}{2}-2\varphi)s}  \,ds \right)^p		  \right] \left( \|\bar B\|^{2p}_{\infty,-\frac{\beta}{4}}  +1  \right)\\
	<&~\infty.
\end{align*}
Thus, it follows that by \eqref{estimate:E/Gamma}
		\[
		\mathbb E\left[\sup_{t\ge 0}\left| e^{\varphi t} \gamma_t {\widetilde{P}_t^{\ast}}  \right|^p\right]\leq \mathbb E\left[\sup_{t\ge 0}e^{2p\varphi t}|\Lambda_t|^{2p}\right]^\frac{1}{2}\mathbb E\left[\sup_{t\ge 0}\left|\Gamma^{-1}_t\hat{\mathcal E}_t\right|^{2p}\right]^\frac{1}{2}<\infty. 
		\]
By Lemma \ref{lemma:gamma} it follows that 
\[
	\mathbb E\left[\sup_{t\ge 0}e^{4\varphi t}\gamma_t^2(\widetilde{P}_t^{\ast})^4\right]\leq\mathbb E\left[\sup_{t\ge 0} \left| e^{\varphi t} \gamma_t {\widetilde{P}_t^{\ast}}  \right|^8\right]^\frac{1}{2}\mathbb E\left[\sup_{t\ge 0} \frac{1}{\gamma_t^4}\right]^\frac{1}{2}<\infty.
\]   
Finally, by \eqref{optimal-state-inf} and Assumption \ref{ass:existence-control} it holds
\begin{equation*}
	\begin{aligned}
		&~\mathbb E \left[ \sup_{t\geq 0} e^{2\varphi t}\gamma_t (\widetilde{X}_t^*)^2  \right] \\
		\leq&~ c \mathbb E \left[ \sup_{t\geq 0} e^{2\varphi t} \left(1-\pi^{\alpha_t}_t\right)^2 \gamma_t (\widetilde{P}_t^*)^2  \right]  + c \mathbb E \left[ \sup_{t\geq 0} e^{2\varphi t} \gamma_t(\eta^{\alpha_t}_t)^2  \right] \\
		\leq&~c \mathbb E \left[ \sup_{t\geq 0} e^{4\varphi t} \gamma_t^2 (\widetilde{P}_t^*)^4  \right]^{1/2}  \left(1+\mathbb E \left[ \sup_{t\geq 0} \left|\pi^{\alpha_t}_t\right|^4   \right]^{1/2} \right)+ c \mathbb E \left[ \sup_{t\geq 0}  \gamma_t^2 \right]^{1/2}\mathbb E \left[ \sup_{t\geq 0} e^{4\varphi t}(\eta^{\alpha_t}_t)^4  \right]^{1/2}	\\
		<&~ \infty.
	\end{aligned}
\end{equation*}
\end{proof}

\begin{proof}[Proof of Theorem \ref{thm:existence-optimal-control}]

	 By the identity \eqref{eq:cost-inf} in Theorem \ref{thm:reformulation-cost-inf}, it holds 
	\begin{equation*}
		\begin{split}
			\mathbb{E}_t\left[    \int_{t}^{\infty} e^{-\phi s}\left(-Y_{s-}\,dX_s + \frac{\gamma_s}{2}\,d[X]_s -\sigma^{\alpha_s}_s \,d[X,W]_s+\lambda^{\alpha_s}_s X_s^2\,ds \right)  \right]
			\geq\mathcal X^\top_{t-}A_t^{i}\mathcal X_{t-} + \mathcal X^\top_{t-}B^{i}_t+C^{i}_t.
		\end{split}
	\end{equation*}  
	Moreover, it is easy to check that $I^{\bar A,i}\mathcal X^{*}+I^{\bar B,i}=0$ for ${\mathcal X^{*}}:=(\widetilde X^*,\widetilde Y^*)^\top$ defined in \eqref{optimal-state-inf}, for any $i\in \cal M$. Thus, $\widetilde X^*$ in \eqref{optimal-state-inf} and \eqref{jump-inf} is optimal if it is a c\`adl\`ag semimartingale.

	It remains to verify the initial jump \eqref{jump-inf} is optimal. Note that after the initial jump, it should hold $I^{\bar A,i}_0\mathcal X_0+I^{\bar B,i}_0=0$, equivalently, we have 
		$$
		I^{\bar A,i}_{1,0} \widetilde X_0+I^{\bar A,i}_{2,0}\widetilde Y_0+I^{\bar B,i}_0=I^{\bar A,i}_{1,0} ( X_{0-}+\Delta X_0)+I^{\bar A,i}_{2,0}( Y_{0-}  + \Delta Y_0 )+I^{\bar B,i}_0=0,
		$$
		where $I^{\bar A}_1$ and $I^{\bar A}_2$ denote the first and the second components of $I^{\bar A}$. 
		 Note that $\Delta Y_0=-\gamma_0 \Delta X_0$ and that $\gamma I^{\bar A}_2-I^{\bar A}_1=a$. It holds that $\Delta X_0=\frac{1}{a_0^{\alpha_0}}\left(  I^{\bar A,\alpha_0}_0 \mathcal X_{0-} + I_0^{\bar B,\alpha_0}    \right)$.
\end{proof}

\section{Conclusion}
In this paper, we study a class of infinite‑horizon stochastic control problems with regime switching arising from optimal liquidation with semimartingale strategies. We characterize the value function and the optimal strategy through three systems of BSDEs defined on an infinite time horizon. Existence of solutions to these BSDE systems is established through the development of new analytical techniques involving BMO analysis and comparison principles for multidimensional BSDEs.
The analysis of the stochastic control problem proceeds in two steps. First, we establish the well‑posedness of the control problem by rewriting the cost functional with the aid of the BSDE characterization. Second, under additional assumptions, we prove the existence of an optimal control. 
 
 \begin{appendix}
 	\section{Results on $A$.}\label{app:A}

  By extending the proof in   \cite[Theorem 4.2]{Hu-COCV}, we obtain the following wellposedness result for $\bar A$.
 	\begin{proposition}\label{prop:A}
 		The system of infinite-horizon BSDEs \eqref{BSDE:A-inf} admits a solution  $(\bar A,Z^{\bar A})$ in $L^{\infty,0} \times H^{2,K_{\bar A}}_{BMO} $ for each $K_{\bar A}>0$. Moreover,  	$\bar A^{i}$ is valued in $[0,1/2]$ for each $i\in\mathcal M$.
 	\end{proposition}
 	\begin{corollary}\label{coro:gamma-a}
 	Recall $a^{i}$ is defined in \eqref{def:a-i}. It satisfies $0<\epsilon\leq \frac{a^{i}}{\gamma}\leq c$. 
 	\end{corollary}
 	\begin{proof}
 		By definition, we have
 			\[
 				\frac{a^{i}}{\gamma} = \sigma_1^2\bar A^{i} +\frac{\mu}{2}-\frac{\sigma_1^2}{2} +\rho^i+\frac{\phi}{2}+\frac{\lambda^i}{\gamma}.
 		\]
 		By Proposition \ref{prop:A}, Assumption \ref{ass:standing} (i) and (iii), we have $\frac{a^{i}}{\gamma}\geq \epsilon$. By Proposition \ref{prop:A} and Assumption \ref{ass:standing} (i) and (ii), 
 		$$
 			 \sigma_1^2\bar A^{i} +\frac{\mu}{2}-\frac{\sigma_1^2}{2} +\rho^i+\frac{\phi}{2}+\frac{\lambda^i}{\gamma} \leq \frac{\mu}{2}  +\rho^i+\frac{\phi}{2}+\frac{\lambda^i}{\gamma}\leq c<\infty.
 		$$ 
 	\end{proof}

 \end{appendix}

\bibliographystyle{plain}
\bibliography{bib_CFX2026}

\end{document}